\documentclass[11pt,reqno]{amsart}

\usepackage{amsmath,amsfonts,amssymb,mathrsfs,stackengine}
\usepackage{amssymb,mathrsfs,graphicx,extpfeil}
\usepackage{amsmath,amsfonts,amssymb,amscd,amsthm,bbm}
\usepackage{epsfig}

\usepackage{graphicx,colortbl,here}
\usepackage{extpfeil}
\usepackage{subcaption}
\usepackage{caption}
\usepackage{kotex}
\usepackage{hyperref}
 \usepackage {cite}
\newcommand{\bbr}{\mathbb R}

\newcommand*\di{\mathop{}\!\mathrm{d}}

\newcommand{\opnorm}[1]{{ \vert\kern-0.25ex \vert\kern-0.25ex \vert #1 
   \vert\kern-0.25ex \vert\kern-0.25ex \vert}}

\newtheorem{theorem}{Theorem}[section]
\newtheorem{lemma}{Lemma}[section]
\newtheorem{corollary}{Corollary}[section]
\newtheorem{proposition}{Proposition}[section]
\newtheorem{remark}{Remark}[section]
\newtheorem{definition}{Definition}[section]

\hypersetup{
	hyperindex=true
}
\hypersetup{hidelinks}
\allowdisplaybreaks[4]
\begin{document}

\title[Weak flocking of kinetic CS model ]{On the weak flocking of the kinetic Cucker-Smale model in a fully non-compact support setting} 

\author[Ha]{Seung-Yeal Ha}
\address[Seung-Yeal Ha]{\newline Department of Mathematical Sciences and Research Institute of Mathematics, \newline
	Seoul National University, Seoul, 08826, Republic of Korea}
\email{syha@snu.ac.kr}

\author[Wang]{Xinyu Wang}
\address[Xinyu Wang]{\newline Department of Mathematical Sciences, \newline
	Seoul National University, Seoul, 08826, Republic of Korea}
\email{wangxinyu97@snu.ac.kr}

\thanks{\textbf{Acknowledgment.} 
The authors are deeply indebted and grateful to the anonymous reviewers whose extensive comments and suggestions have substantially improved the quality of this paper.
	The work of S.-Y. Ha is supported by National Research Foundation (NRF) grant funded by the Korea government(MIST) (RS-2025-00514472), and the work of X. Wang is supported by the Natural Science Foundation of China (grants 123B2003), the China Postdoctoral Science Foundation (grants 2025M774290), and Heilongjiang Province Postdoctoral Funding (grants LBH-Z24167).}

\begin{abstract}
We study the emergent behaviors of the weak solutions to the kinetic Cucker-Smale (in short, KCS) model in a non-compact spatial-velocity support setting. Unlike the compact support situation, non-compact support of a weak solution can cause a communication weight to have a zero lower bound, and position difference does not have a uniformly linear growth bound. These cause previous approach based on the nonlinear functional approach for spatial and velocity diameters to break down. To overcome these difficulties, we derive refined estimates on the upper bounds for the high-order spatial-velocity moments and show the uniqueness of the weak solution using the estimate on the deviation of particle trajectories. For the estimate of emergent dynamics, we consider two classes of distribution functions with decaying properties (an exponential decay or polynomial decay) in phase space, and then develop an {\it time-varying effective region} approach to verify that the second moment for the velocity deviation from an average velocity tends to zero asymptotically, while the second moment for spatial deviation from the center of mass remains bounded uniformly in time. This illustrates the robustness of the weak flocking dynamics of the KCS model even for fully non-compact support settings in phase space and generalizes earlier results on strong flocking dynamics in a compact support setting.
\end{abstract}



\maketitle


\section{Introduction}\label{sec:1}
\setcounter{equation}{0}
The dynamics of the Cucker-Smale model with a large system size has been effectively approximated by the corresponding kinetic model via the mean-field limit  \cite{kinetic4,k4,k2,kinetic2,ks2,kinetic5,w4}. The resulting kinetic model belongs to the class of the Vlasov-McKean type models, which is a nonlinear transport equation on the phase space with a non-local forcing term. In most literature, we deal with a situation in which physical observables are confined in a bounded region in the phase space $\bbr_x^d \times \bbr_v^d =\bbr^{2d}$ so that the kinetic density (weak solution to the KCS model) has a compact support in the phase space, and flocking dynamics has also been considered in the same  compact support setting using the projected spatial and velocity support of the kinetic function and second velocity moment (see \cite{k4,k2, k1, w4}). Recently, the second author and his collaborators \cite{w6, w2} considered the KCS model in a spatially extended setting so that the spatial support of the kinetic density can be unbounded, whereas the velocity support is still bounded. Thus, whether the flocking dynamics of the KCS model can emerge from a non-compact velocity support setting is an interesting remaining question until now.

In this paper, we revisit this remaining question with a final resolution. More precisely, we provide a sufficient framework leading to emergent dynamics of the KCS model in a fully non-compact setting in phase space. To set up the stage, we begin with the CS model \cite{cucker2}. Let $x_i$ and $v_i$ be the position of the $i$-th CS particle. The (particle) CS model reads as follows.
\begin{equation} \label{A-0}
\begin{cases}
\displaystyle  {\dot x}_i = v_i, \quad t > 0,~~i = [N]:= \{ 1, \cdots, N \}, \vspace{6pt}\\
\displaystyle {\dot v}_i = \frac{\kappa}{N} \sum_{j \in [N]} \phi(|x_j - x_i|) (v_j - v_i),
\end{cases}
\end{equation}
where $\kappa$ and $| \cdot |$ are the nonnegative coupling strength and the standard Euclidean norm in $\bbr^d$, respectively, and $\phi$ denotes the long-ranged communication weight taking the following explicit form:
\begin{equation}\label{A-0-1}
	\phi(r)=\frac{1}{({1+ r^2})^{\frac{\beta}{2}}},\quad r \geq 0, \quad  0 \leq \beta \leq 1. 
\end{equation} 

Note that the spatial decay rate of $\phi$ is $\beta$ and communication between particles is assumed to be long-ranged $(\beta \in [0, 1])$. For a short-ranged communication weight with $\beta > 1$, local-flocking can emerge asymptotically which is beyond of the scope of this paper (see \cite{bi-cluster}). The CS model \eqref{A-0}--\eqref{A-0-1}  has been extensively studied from various perspectives, e.g., collision avoiding \cite{collision1,collision2}, flocking with hierarchical leadership \cite{leader,ex5}, discrete flocking \cite{DS1,DS2}, rooted leadership flocking \cite{switch, L12}, multi-cluster flocking \cite{bi-cluster}, stochastic flocking \cite{ks1,ks2,CM-2008}, infinite particle \cite{w1,w3, J1,w5}, etc. When the system size $N$ of \eqref{A-0} is sufficiently large, we can use the mean-field approximation to describe the dynamics of the one-particle distribution function which can be realized as the suitable weak limit of the empirical measure associated with particle configuration. More precisely,  let $f = f(t, x, v)$ be the one-particle distribution function (or simply kinetic density) for the Cucker-Smale ensemble at  phase space position $(x,v)$ at time $t$. Then, the temporal-phase space dynamics of $f$ is governed by the following Cauchy problem for the KCS model:

\begin{equation} 
	\begin{cases} \label{A-1}
		\displaystyle \partial_tf+ \nabla_x \cdot (v f) + \nabla_v \cdot (L[f]f)=0, \quad t > 0,~~(x, v) \in {\mathbb R}^{2d},  \vspace{6pt}\\
		\displaystyle L[f](t,x,v)=-\kappa \int_{\mathbb{R}^{2d}} \phi( |x-x_{\star}|) \left(v-v_{\star}\right)  f(t, x_{\star}, v_{\star}) \di x_{\star} \di v_{\star}, \\
		\displaystyle f \Big|_{t = 0}  =f^{\mathrm{in}}.
	\end{cases}
\end{equation}

\noindent The KCS model \eqref{A-1} can be rigorously derived from the mean-field limit  \cite{k2, k1, HKPZ, kinetic2, NP} from \eqref{A-0}.  When the initial datum $f^{\mathrm{in}}$ is compactly supported in $x$ and $v$, the global existence, uniqueness, and stability results for \eqref{A-1} have been extensively studied in \cite{k4, kinetic1, kinetic2, kinetic4, kinetic5,k2, k12, k1, kinetic3, a2}, and the flocking behavior was also studied in \cite{k2,k4} in a strong sense. The authors in the  aforementioned works showed that there exist positive constants $C$ and $\Lambda$ depending only on ${\rm spt}f^{\mathrm{in}}$ (support of the initial kinetic density $f^{\mathrm{in}}$) and $\beta$ such that
\begin{equation}
\begin{cases} \label{A-2}
\displaystyle \sup\limits_{0\le t<\infty}\sup\limits_{(x,v)\in{\rm spt} f(t)} |x- (x_c^{\mathrm{in}} + t v_c^{\mathrm{in}}) |\le C: \quad  \mbox{spatial cohesion},\vspace{6pt}  \\
\displaystyle \sup\limits_{(x,v)\in{\rm spt} f(t)} |v-v_c^{\mathrm{in}} |\le Ce^{-\Lambda t}: \hspace{2cm}  \mbox{velocity alignment},
\end{cases}
\end{equation}
where $x^{\mathrm{in}}_c$ and $v^{\mathrm{in}}_c$ are the center of spatial and average velocity of $f^{\mathrm{in}}$:

\[  x_c^{\mathrm{in}} := \frac{1}{\| f^{\mathrm{in}}\|_{1}}\int_{\mathbb{R}^{2d}}xf^{\mathrm{in}}(x, v) \di x \di v, \quad v_c^{\mathrm{in}} := \frac{1}{\| f^{\mathrm{in}}\|_{1}}\int_{\mathbb{R}^{2d}}vf^{\mathrm{in}}(x, v) \di x \di v. \]
Here, $\| \cdot \|_{1}$ is the $L^1$-norm on $\bbr^{2d}$. \newline

Note that the first estimate in \eqref{A-2} illustrates that the spatial support of kinetic density lies in a bounded region around the center of mass (spatial cohesion), whereas the second estimate in \eqref{A-2}
 denotes that the velocity fluctuations around the velocity average tend to zero (asymptotic formation of velocity alignment). Hence both estimates provide information on the projected spatial and velocity supports of the kinetic density. In this sense, we call the estimates \eqref{A-2} as a strong flocking dynamics for \eqref{A-1}. \newline 
 
In contrast to the well-studied strong flocking theory for compactly supported initial data, in this work, we consider the KCS model in a fully non-compact spatial-velocity setting. The study of non-compact initial data has a long history in kinetic theory, especially for Vlasov-type equations over the past several decades, including the Vlasov-Poisson system \cite{C-Z-2016,L2016}, the Vlasov-Maxwell system \cite{D-L1989,S2004}, and the references therein. In particular, in kinetic flocking models with diffusion \cite{ks1,D-F-T-2010}, both the spatial and velocity supports generally become non-compact for any positive time due to the diffusive effect, even when the initial data are compactly supported. These observations indicate that the non-compact setting is both mathematically natural and physically important. In such a fully non-compact setting, the standard strong flocking estimates \eqref{A-2} is no longer expected to hold, since the spatial diameter is no longer uniformly controlled. Fortunately, we can still use nonlinear functional approach for velocity alignment based on the second moment for the velocity fluctuation around the average one as in \cite{k1}:
 \[
 \int_{\bbr^{2d}} |v-v_{c}|^2 f (t,x,v) \di x \di v, \qquad \mbox{where}~~\quad v_c := \frac{1}{\| f \|_{1}} \int_{\bbr^{2d}} v f(t,x,v) \di x \di v.
 \]
The asymptotic behavior of the KCS model \eqref{A-1} on a non-compact domain was first addressed by Chen and Yin \cite{k6} about a decade ago on a tubular domain in phase space, and they showed that if the initial datum $f^{\mathrm{in}}$ satisfies the boundedness condition in the sense that there exists a positive constant $\lambda$ such that 
\begin{equation} \label{A-3}
	\sup\left\{\left |v-\frac{x}{\lambda}\right |:(x,v)\in {\rm spt}f^{\mathrm{in}}\right\}< \infty,
\end{equation}
then they derived a new type of collective behavior for a short-ranged communication weight with $\beta > 1$:
\begin{equation} \label{A-4}
	\lim_{t \to \infty} \int_{\mathbb{R}^{2d}}\left |v-\frac{x}{t+\lambda}\right |^k f(t, x, v) \di x \di v = 0, \quad \forall\ k\ge 2.
\end{equation} 
Recently, the authors in \cite{w2,w6} showed that even for the cases where the spatial support of the initial datum $f^{\mathrm{in}}$ is unbounded, the second velocity moments centered around the initial mean velocity  tend to zero asymptotically for a long-ranged communication with $\beta\in[0, 1]$. So far, all previous works on the asymptotic behaviors of the KCS model require that initial datum takes either compact velocity support or a tubular type support in phase space. Based on what we have discussed so far, it is natural to ask the following two questions: 
\begin{quote}
	\begin{itemize}
		\item
		(Q1):~When initial kinetic density is not compactly supported in spatial-velocity variables, is the weak solution to \eqref{A-1} unique?
		\vspace{0.2cm}
		\item
		(Q2):~If the weak solution is unique, under what conditions on initial datum and system parameters, can we guarantee a flocking dynamics for the weak solution to the KCS model?
	\end{itemize}
\end{quote}
In this paper, we answer aforementioned two questions in an affirmative manner. For the first question (Q1), we estimate temporal evolution of a functional measuring deviations between particle trajectories issued from the same phase space point. 

Consider the initial datum $f^{\mathrm{in}}$ satisfying the following positivity, boundedness and integrability conditions:
\[
f^{\mathrm{in}} \in(L^1 \cap L_+^\infty)(\mathbb{R}^{2d}) \quad \mbox{and} \quad  ( |x|^2 + |v|^2)f^{\mathrm{in}} \in L^1(\mathbb{R}^{2d}).
\]
In this situation, the global existence of weak solutions to the KCS model has already been obtained in \cite{kinetic3} without uniqueness. To establish the uniqueness of the weak solution, we need to additionally require that initial data exhibit exponentially decay on velocity variable, i.e., there exists a positive constant $\alpha>0$ such that
\begin{align}\label{New1-7}
	\int_{\mathbb{R}^{2d}} e^{\alpha |v|} {f^{\mathrm{in}}(x, v) \di x \di v}<\infty.\end{align}
Let $f$ and $g$ be two weak solutions to \eqref{A-1} corresponding to the same initial datum $f^{\mathrm{in}}$, respectively (see Definition \ref{D2.1} for the concept of weak solution), and let $(X_{f}(t),V_{f}(t))$ and $(X_{g}(t),V_{g}(t))$ be the particle trajectories issued from $(x,v)$ at time $t = 0$, respectively (see Section \ref{sec:2.1}). Then, we introduce a functional $\Delta(f,g)$: 
\[
	\Delta(f,g)(t):=\int_{\mathbb{R}^{2d}} \Big ( |X_{f}(t)-X_{g}(t) |+ |V_{f}(t)-V_{g}(t)| \Big ){f^{\mathrm{in}}(x, v) \di x \di v}, \quad \forall~t\in[0,\tau].
\]
Note that $\Delta(f,g)$ measures overall deviations between two particle trajectories $(X_{f}(t),V_{f}(t))$ and $(X_{g}(t),V_{g}(t))$. Then, we derive the following integral inequality:
\begin{align}\label{New1-8}
\begin{cases}
\displaystyle \Delta[f,g](t) \leq C\int_0^t \Delta[f,g](s)\di s +  C\int_0^t\left(\int_{\mathbb{R}^{2d}} |V_{f}(s)||X_{f}(s)-X_{g}(s)|{f^{\mathrm{in}}(x, v) \di x \di v}\right)\di s, \vspace{6pt}\\
\displaystyle \Delta(f,g)(0) = 0.
\end{cases}
\end{align}
Furthermore, we combine \eqref{New1-7}, \eqref{New1-8}, and $\Delta(f,g)(0) = 0$ to obtain
\[ \Delta(f,g)(t) = 0, \quad t > 0, \]
which implies the uniqueness of weak solution:
\[ X_f = X_g,\quad \mbox{and} \quad  V_f = V_g, \quad \mbox{i.e.,} \quad f = g. \]
(see Theorem \ref{T2.1} and Section \ref{sec:3.2} for details).  \newline

Next, we return to the question (Q2).  First, we recall the concept of weak flocking as follows.
\begin{definition}\label{D1.1}
\emph{(Weak flocking)}
For $\tau  \in (0,\infty]$, let $f\in L^{\infty}([0, \tau);  L^1(\mathbb{R}^{2d}))$ be a weak solution to \eqref{A-1} with bounded second moments in position and space:
\[ \int_{\bbr^{2d}} (|x|^2 + |v|^2) f(t,x,v) \di x \di v < \infty.
\]
Then, $f$ exhibits weak flocking asymptotically if and only if the following two relations hold.
\begin{equation}
\begin{cases} \label{A-5}
\displaystyle \sup\limits_{0\le t<\infty}\int_{\mathbb{R}^{2d}} |x-x_c^{\mathrm{in}} -v_c^{\mathrm{in}} t |^2f(t, x, v) \di x \di v < \infty, \vspace{6pt}\\
\displaystyle  \lim_{t \to \infty} \int_{\mathbb{R}^{2d}} |v-v_c^{\mathrm{in}} |^2f(t, x, v)  \di x \di v  = 0. 
\end{cases}
\end{equation}
\end{definition}
\noindent The estimates in \eqref{A-5} correspond to the relaxed concept of flocking dynamics, and it was used in \cite{k1, w6} as a measurement for flocking, and in the aforementioned works, the authors still require at least the velocity support to be bounded. Compared to the strong flocking \eqref{A-2}, weak  flocking in Definition \ref{D1.1} focuses on the evolutions of second spatial-velocity moments, which are suitable for a fully non-compact support framework. 

Before we move on to the presentation of the second main result, we briefly discuss possible difficulties of fully non-compact setting with the particle CS model \eqref{A-0}. A distinguishing feature of \eqref{A-0} is that the velocity alignment force diminishes as the relative position differences between particles tend to infinity. Consequently, most existing research on the CS-type models focuses on the verification that the relative positions between particles remain uniformly bounded, ensuring that the communication weight between particles retains a positive lower bound (see \cite{k4, Cucker, cucker2, k2}). Furthermore, many studies directly assume that the communication weight between particles has a positive lower bound to avoid such challenge of proving bounded inter-particle distances (see \cite{ks2, ks1}). Additionally, for weak solutions with a compact velocity support, the growth of relative position differences is at most linear (see \cite{w2,w6}). However, for initial data without compact spatial-velocity support, the lower bound of the inter-particle velocity alignment force in the CS model becomes zero, and the most of mass may not be concentrated on the linearly growing region in phase space. To overcome these issues, we impose two types of decay conditions (exponential vs. polynomial) on initial kinetic density at infinity: for positive constants $\alpha>0, \delta\ge1,~{D:=\min\{D_1, D_2\} \geq 2}$:
\begin{align}
	\begin{aligned}  \label{A-6}
	        & \int_{\mathbb{R}^{2d}} \Big( |x|^{D_1}+ |v|^{D_2} \Big) {f^{\mathrm{in}}(x, v) \di x \di v}=: {\mathcal M}_p(f^{\mathrm{in}}, D_1, D_2) < \infty: \hspace{0.3cm} \mbox{Slow decaying class}, \vspace{8pt}\\
		& \int_{\mathbb{R}^{2d}} e^{\alpha( |x|+|v|^{\delta})} {f^{\mathrm{in}}(x, v) \di x \di v}=: {\mathcal M}_e(f^{\mathrm{in}}, \alpha, \delta) < \infty: \hspace{1.8cm} \mbox{Fast decaying class}.
\end{aligned}
\end{align}
Here, the subscripts in ${\mathcal M}_e$ and ${\mathcal M}_p$ represent the initial of exponential and polynomial decays, respectively. Our second set of main results establishes the weak flocking dynamics of the weak solutions under some suitable framework in terms of initial data and system parameters. First, we consider system parameters and initial datum: 
\begin{align*}
\begin{aligned} 
&0 \leq \beta < 1, \quad \kappa> 0, \quad  \gamma > 1, \quad  \gamma\beta<1,\quad l_1>1, \quad   D_1 > \frac{2l_1}{(l_1-1)\left(\gamma -1\right)},    \\
&  D_2\ge \max\left\{D_1,~2l_1 \right\},\quad {\mathcal M}_p(f^{\mathrm{in}}, D_1,D_2) :=  \int_{\mathbb{R}^{2d}} (|x|^{D_1}+ |v|^{D_2}) {f^{\mathrm{in}}(x, v) \di x \di v} < \infty.
\end{aligned}\end{align*}
{Here, $\gamma$ is an auxiliary scaling parameter in the time-varying effective region defined in \eqref{NewD-4-4}.} Moreover, we use the high-order moment bounds (see Lemma \ref{L3.1} and Corollary \ref{C3.1}) to give an energy estimate on the time-varying effective region. With these estimates, we show that the weak flocking behavior emerges at least algebraically fast (see Theorem \ref{T2.2} and Section \ref{sec:4.1} for details). \newline

Second, we consider system parameters and initial datum:
\[ 0 \leq \beta < 1, \quad \kappa > 0, \quad \alpha > 0, \quad   {\mathcal M}_e(f^{\mathrm{in}}, \alpha) :=  \int_{\mathbb{R}^{2d}} e^{\alpha \left(| x|+| v |\right)} {f^{\mathrm{in}}(x, v) \di x \di v} < \infty.\]
In this case, weak flocking emerges at least super-polynomial fast, i.e., the convergence is faster than any algebraic rate. Furthermore, if we additionally require that system parameters and initial datum satisfy
\[ \delta>1, \quad \beta\in\Bigg[0,\frac{\delta-1}{\delta}\Bigg),\quad  {\mathcal M}_e(f^{\mathrm{in}}, \alpha, \delta) :=\int_{\mathbb{R}^{2d}} e^{\alpha \left(| x|+| v |^{\delta}\right)} {f^{\mathrm{in}}(x, v) \di x \di v} < \infty.\]
Then, we can show that weak flocking emerges at least exponentially fast (see Theorem \ref{T2.2} and Section \ref{sec:4.2}  for details). 
In particular, if we further assume that the solution has compact velocity support, we can still establish that the weak flocking  will emerge at least algebraically fast even for critical exponent $\beta=1$ (see Theorem \ref{T2.2} and Section \ref{sec:4.3} for details). \newline

The rest of this paper is organized as follows. In Section \ref{sec:2}, we recall the basic properties of the weak solution to \eqref{A-1} and  previous results, and then we summarize our main results. In Section \ref{sec:3}, we provide a rigorous proof of the uniqueness of the weak solution.  In Section \ref{sec:4}, we establish weak flocking behaviors \eqref{A-5} of the KCS model for  two types of distribution classes described in \eqref{A-6}. Finally, Section \ref{sec:5} is devoted to a brief summary of our main results and some remaining issues for a future work.

 \vspace{1cm} 
 
\noindent\textbf{Gallery of Notation:} Let $\mu$ be a measure defined on the Euclidean space $\bbr^{n}$. Moreover, we set
\begin{align*}
\begin{aligned}
	\mathcal{P}(\mathbb{R}^{n}) &:= \left\{\mbox{the set of all probability measures $\mu$ defined on $\bbr^n$}\right\}, \\  
	\mathcal{P}_m(\mathbb{R}^{n}) &:=\left\{\mu \in\mathcal{P}(\mathbb{R}^{n}):\int_{\mathbb{R}^{n}} |z|^m \mu (\di z)<\infty\right\}, \quad \mbox{for $m \in  \{ 0\} \cup {\mathbb N}$}.
\end{aligned}
\end{align*}
For $\mu \in {\mathcal P}(\bbr^{2d})$, if $\mu$ is absolutely continuous with respect to the Lebesgue measure $\di x \di v$ and  has a probability density function $f \in L^1(\bbr^{2d})$, we write it as 
\[ \mu(\di x, \di v) = f(x,v) \di x \di v. \]
Let ${\mathcal C}(\Omega),~{\mathcal C}_b(\Omega),~{\mathcal C}_c^1(\Omega) $ and $L^p(\Omega)$ be the  spaces of all continuous functions, bounded and continuous functions, continuously differentiable functions with compact supports and $L^p$-integrable functions on $\Omega$, respectively.  We denote $| \cdot |$ and $\| \cdot \|_p$ by the Euclidean norm on $\bbr^{d}$ and $L^p(\Omega)$, respectively. We also denote $L^\infty_+(\Omega)$ by the space of all bounded and nonnegative measurable functions defined on $\Omega$.  \newline

For $f$ and $g$, we say that $f \lesssim g$ if and only if there exists a positive constant $C$ such that $f \leq C g$, and throughout the paper, we use abbreviated notation:
\[ z = (x,  v), \quad  z_\star = (x_\star,  v_\star), \quad   \di z = \di x \di v, \quad  \di z_\star = \di x_\star \di v_\star.  \]

%
%
%

\section{Preliminaries}\label{sec:2}
\setcounter{equation}{0}
In this section, we review basic materials and previous results for the KCS model, and then we delineate main results on the uniqueness of weak solution and emergence of weak flocking whose proofs can be provided in later sections.  

\subsection{Preparatory materials} \label{sec:2.1}
In this subsection, we recall all necessary materials to be used in later sections. First, we recall the concept of weak solution to \eqref{A-1}.
\begin{definition}\label{D2.1}
\emph{(Weak solution)}
For $\tau  \in (0,\infty]$, let $f\in {\mathcal C}([0, \tau);  L^1(\mathbb{R}^{2d}))$ be a  nonnegative weak solution to  \eqref{A-1} with the initial datum $f^{\mathrm{in}} \in (L^1 \cap L_+^{\infty})(\mathbb{R}^{2d})$ if the following relations hold:
	\begin{enumerate}
		\item
		$f$ is weakly continuous in $t$:~$\forall~\psi \in {\mathcal C}_c^1(\mathbb{R}^{2d})$, 
		\[
		t \quad \mapsto \quad \int_{\mathbb{R}^{2d}}\psi f \di z \quad\text{is continuous}.
		\]
		\item	
		$f$ satisfies \eqref{A-1} in weak sense:~$ \forall~\zeta \in {\mathcal C}_c^1([0, \tau) \times\mathbb{R}^{2d}),$
\[ \int_{\mathbb{R}^{2d}}\zeta(t, z) f(t, z) \di z =  \int_{\mathbb{R}^{2d}}\zeta(0,z) f^{\mathrm{in}}(z) \di z + \int_0^{\tau} \int_{\mathbb{R}^{2d}}(\partial_s\zeta+v\cdot \nabla_x \zeta +\nabla_v\zeta \cdot L[f])f(s, z) \di z \di s.\]
	\end{enumerate}	
\end{definition}
\vspace{0.2cm}

Next, we recall the concept of $p$-Wasserstein space. Let $\mu$ and $\nu$ be probability measures on $\bbr^{d}$. Then, for  all nonnegative integrable functions $\psi_1$ and $\psi_2$ defined on $\mathbb{R}^{d}$, we define a probability measure $\pi \in {\mathcal P}(\bbr^{d} \times \bbr^d)$ with marginals $\mu$ and $\nu$ as follows.
	\begin{equation*}\label{f3.1}
		\int_{\mathbb{R}^{d}\times\mathbb{R}^{d}}(\psi_1(x)+\psi_2(x_{\star}))\pi(\di x,\di x_{\star})=\int_{\mathbb{R}^{d}}\psi_1(x) \mu(\di x)+\int_{\mathbb{R}^{d}}\psi_2(x_{\star}) \nu(\di x_{\star}).
	\end{equation*}
	We denote $\Pi(\mu, \nu)$ by the set of probability measures on $\bbr^d \times \bbr^d$ with marginals $\mu$ and $\nu$:
	\[ \mu(\di x) = \int_{\bbr^d} \pi(\di x, \di x_\star) \di x_\star, \quad  \nu(\di x_\star) = \int_{\bbr^d} \pi(\di x, \di x_\star) \di x. \]

Next, we recall the basic properties of probability measures in the following definition. 
		\begin{definition}\label{D2.2}
		\emph{($p$-Wasserstein space)}
	 For any $\mu, \nu \in \mathcal{P}(\mathbb{R}^{d})$,  the Wasserstein distance of order $p$ between $\mu$ and $\nu$ (or $p$-Wasserstein distance between $\mu$ and $\nu$) is defined by the formula:
	\begin{equation*}
		W_p(\mu, \nu):=\inf\left\{\left(\int_{\mathbb{R}^{d}\times\mathbb{R}^{d}} |x-x_{\star}|^p\pi(\di x,\di x_{\star})\right)^\frac{1}{p}:\pi\in\Pi(\mu, \nu)\right\},
	\end{equation*}
	where $\Pi(\mu,\nu)$ is the set of all probability measures on $\mathbb{R}^{d}\times\mathbb{R}^{d}$ with marginals $\mu$ and $\nu$ in ${\mathcal P}(\bbr^d)$, respectively. 	
\end{definition}
\begin{remark}
In order to avoid the difficulty that $W_p$ may take the value $+\infty$, we consider $W_m$ on $\mathcal{P}_m(\mathbb{R}^{d})$ for a suitable $m \in {\mathbb N}$.
\end{remark}
In the following lemma, we provide the characterization of the convergence in $W_p$.
\begin{lemma}{\rm \cite{k8}} \label{le2.5}
The following statements hold. 
\begin{enumerate}
\item
The metric space $(\mathcal{P}_p(\mathbb{R}^{d}), W_p)$ is complete.
\vspace{0.1cm}
\item
 Let $(\mu_k)_{k\in\mathbb{N}}$ be a sequence of probability measures in $\mathcal{P}_p(\mathbb{R}^{d})$, and let $\mu$ be another element of $\mathcal{P}_p(\mathbb{R}^{d})$.  Then, the following statements hold.
 \vspace{0.1cm}
 \begin{enumerate}
 \item
  The sequence $(\mu_k)_{k\in\mathbb{N}}$ has uniformly integrable $p$-th moments if for some $x_0\in\mathbb{R}^d$$:$
	\begin{equation*}
		\lim\limits_{r\rightarrow\infty}\int_{\mathbb{R}^d\setminus B_r(x_0)} |x-x_0|^p \mu_k(\di x)=0\quad \text{uniformly with respect to }k\in\mathbb{N},
	\end{equation*}
where $B_r(x_0)$ is the open ball with a center $x_0$ and a radius $r$. 	
\vspace{0.1cm}
\item	
The sequence $(\mu_k)$ converge weakly to $\mu$ if the following relation holds:
	\[
		\lim\limits_{k\rightarrow\infty}\int_{\mathbb{R}^d}\psi(x) \mu_k(\di x)=\int_{\mathbb{R}^d}\psi(x) \mu(\di x), \quad \forall~\psi \in {\mathcal C}_b(\mathbb{R}^d).
	\]
\item
Equivalent relation for the convergence in $W_p$:
	\begin{equation*}
	\lim\limits_{k\rightarrow\infty}W_p(\mu_k, \mu)=0  \quad \Longleftrightarrow \quad
		 \left\{\begin{aligned}
			&\mu_k \quad\mbox{converge weakly to}\quad \mu \quad \mbox{and} \\
			&(\mu_k)_{k\in\mathbb{N}} \quad \mbox{has uniformly integrable p-th moments.}
		\end{aligned}\right.
	\end{equation*}
\end{enumerate}	
\end{enumerate}	
\end{lemma}
\begin{remark}\label{D2.3}
	\emph{(Support of measure and push-forward measure)} Let $\mu$ be a Borel measure on $\mathbb{R}^{2d}$. 
	\begin{enumerate}
		\item
		The support of  $\mu$ is the closure of the set $\{(x,v)\in\mathbb{R}^{2d}:~\mu(B_r(x,v))>0, \forall\ r>0\}$, and we denote it by $\mbox{\rm spt}(\mu)$. 
		\vspace{0.2cm}
		\item
		Let $T: \mathbb{R}^{d}\rightarrow\mathbb{R}^{d}$ be a measurable map. Then, the push-forward measure of $\mu$ by $T$ is the measure $T\# \mu$ defined by $T\# \mu(B)=\mu(T^{-1}(B))$, for all Borel set $B \subset\mathbb{R}^{d}$.
	\end{enumerate}
\end{remark}

\vspace{0.2cm}

Next, we describe a particle trajectory associated with the KCS model. In literature, particle trajectory is also called "bi-characteristics or simply characteristics". For $z = (x, v) \in {\mathbb R}^d \times {\mathbb R}^d$, we define a particle trajectory 
\[ (X(t), V(t)) : = (X(t,0,z), V(t,0,z)) \]
as a unique solution to the following Cauchy problem:
\begin{equation} \label{B-1}
	\begin{cases}
		\displaystyle \dot X(t)=V(t), \quad t > 0,   \\
		\displaystyle \dot V(t)=L[f](t,X(t),V(t)), \\
		\displaystyle X(0)=x, \quad  V(0)=v,
	\end{cases}
\end{equation}
where the linear velocity alignment force along the particle trajectory is given by the following relation:
\begin{equation} \label{B-1-1}
 L[f](t,X(t),V(t))=-\kappa\int_{\mathbb{R}^{2d}} \phi( |X(t) -x_{\star}|) \left(V(t)-v_{\star}\right)  f(t, x_{\star}, v_{\star}) \di x_{\star} \di v_{\star}.
\end{equation}
Note that $L[f]$ in \eqref{B-1-1} is continuous in $t$ and Lipschitz continuous in state variables $(X, V)$, thus the standard Cauchy-Lipschitz theory provides a local well-posedness of  particle trajectories near $t = 0$. Moreover,
 the vector field generated by the right-hand side of \eqref{B-1} is locally bounded in $X$ and sub-linear in $V$, hence particle trajectory is globally well-defined. Thus the particle trajectory map $( X(t,0,x,v), V(t,0,x,v))$ is a well-defined homeomorphism for each fixed time $t$ and a ${\mathcal C}^1$-function of time $t$, and the weak solution $f$ to \eqref{A-1} is given by the push-forward of the initial measure $\mu^{\mathrm{in}}(\di x, \di v) = f^{\mathrm{in}}(x,v) \di x \di v$:
\[ f(t,z)=(X(t,0,z),V(t,0,z))\#f^{\mathrm{in}}, \]
or equivalently,
\begin{equation} \label{B-1-2}
	\int_{\mathbb{R}^{2d}}\psi(z) f(t, z) \di z =\int_{\mathbb{R}^{2d}}\psi(X(t,0,z),V(t,0,z))f^{\mathrm{in}}(z) \di z, \quad \forall~\psi\in {\mathcal C}_b^1(\mathbb{R}^{2d}).
\end{equation}
As a direct application of \eqref{B-1-2}, we have the following elementary estimates. 
\begin{lemma} \label{L2.2}
Let $f = f(t,z)$ be a smooth solution to \eqref{A-1} with sufficiently fast decay at infinite in phase space, and let $(X(t), V(t))$ be an associated particle trajectory defined in \eqref{B-1}. Then, the following estimates hold. 
\begin{align*}
\begin{aligned}
& (i)~  \int_{\mathbb{R}^{2d}}  f(t,z) \di z = \int_{\mathbb{R}^{2d}}  f^{\mathrm{in}}(z) \di z, \quad \int_{\mathbb{R}^{2d}}  v f(t,z) \di z = \int_{\mathbb{R}^{2d}} v f^{\mathrm{in}}(z) \di z. \\
& (ii)~  \int_{\mathbb{R}^{2d}} v f(t,z) \di z = \int_{\mathbb{R}^{2d}} V(t) f^{\mathrm{in}}(z) \di z. \\
& (iii)~  \int_{\mathbb{R}^{2d}} |v|^D f(t,z) \di z = \int_{\mathbb{R}^{2d}} |V(t)|^D f^{\mathrm{in}}(z) \di z  \quad \mbox{for}~D \geq 1. \\
& (iv)~ \int_{\mathbb{R}^{2d}} | v- v_c(t) |^2f(t,z) \di z = \int_{\mathbb{R}^{2d}} |V(t)- v^{\mathrm{in}}_c|^2f^{\mathrm{in}}(z) \di z.
\end{aligned}
\end{align*}
\end{lemma}
\begin{proof}
\noindent (i)~The first relation follows from the direct integration of $\eqref{A-1}_1$ in phase space $\bbr^{2d}$ using the sufficient fast decay conditions. For the second estimate, we multiply $v$ to $\eqref{A-1}_1$ to see
\[
 \partial_t (v f) + \nabla_x \cdot (v \otimes v f) + \nabla_v \cdot (v \otimes L[f]f)= d L[f] f.
\]
We integrate the above relation over $\bbr^{2d}$ using the suitable decay conditions at infinity to get the desired conservation of momentum:
\begin{align*}
\begin{aligned}
\frac{\di}{\di t} \int_{\bbr^{2d}} v f \di z & =d  \int_{\bbr^{2d}}  L[f] f \di z_{\star} \di z \\
&= -d \kappa \int_{\bbr^{4d}} \phi(|x - x_{\star}|) (v-v_{\star})  f(t, z) f(t, z_{\star}) \di z_{\star} \di z \\
& = d \kappa \int_{\bbr^{4d}} \phi(|x - x_{\star}|) (v-v_{\star})  f(t, z) f(t, z_{\star}) \di z_{\star} \di z \\
& = 0. 
\end{aligned}
\end{align*}
Here we use the variable exchange transformation:
\[ (x, v) \quad \Longleftrightarrow \quad (x_{\star}, v_{\star}). \]
 
\noindent (ii) - (iii):~We can take smooth cut-off functions of $v, |v|^D$ as a test function $\psi$ and use \eqref{B-1-2} to see the desired estimates (see Lemma \ref{L3.1} for the smooth cut-off function). In particular, 
\[  \int_{\mathbb{R}^{2d}} v f(t,z) \di z = \int_{\mathbb{R}^{2d}} V(t) f^{\mathrm{in}}(z) \di z, \quad \int_{\mathbb{R}^{2d}} |v|^2 f(t,z) \di z = \int_{\mathbb{R}^{2d}} |V(t)|^2 f^{\mathrm{in}}(z) \di z. \]
(iv) We use all the estimates in (i) - (iii) to see
\begin{align*}
\begin{aligned} 
&  \int_{\mathbb{R}^{2d}} |V(t)- v^{\mathrm{in}}_c|^2 f^{\mathrm{in}}(z) \di z \\
& \hspace{1cm} =  \int_{\mathbb{R}^{2d}} |V(t)|^2  f^{\mathrm{in}}(z) \di z - 2 \Big \langle \int_{\bbr^{2d}} V(t) f^{\mathrm{in}}(z) \di z, v_c^{\mathrm{in}} \Big \rangle + |v_c^{\mathrm{in}}|^2   \int_{\bbr^{2d}}  f^{\mathrm{in}}(z) \di z \\
&  \hspace{1cm}=  \int_{\mathbb{R}^{2d}} |v|^2 f(t,z) \di z -2 \Big \langle  \int_{\mathbb{R}^{2d}} v f(t,z) \di z, v_c^{\mathrm{in}}\Big \rangle   + |v_c^{\mathrm{in}}|^2  \\
&  \hspace{1cm}= \int_{\mathbb{R}^{2d}} | v- v_c^{\mathrm{in}} |^2f(t,z) \di z  =  \int_{\mathbb{R}^{2d}} | v- v_c(t)|^2 f(t,z) \di z.
\end{aligned}
\end{align*}
\end{proof}

\subsection{Previous results} \label{sec:2.2}
In this subsection, we summarize the previous results of  the KCS model in the following two propositions. 
\begin{proposition} 
\emph{\cite{k4,k1,w2,w6}} \label{le:2.3}
For $\tau \in (0, \infty]$, suppose that the initial data $f^{\mathrm{in}},g^{\mathrm{in}}\in (L^1 \cap L^{\infty}_+)(\mathbb{R}^{2d})$ have compact support in velocity variable, i.e., there exists a positive constant $P_\infty$ such that 
\begin{equation*}\label{eq:4.312}
	P_\infty :=\max\left\{\sup\limits_{(x,v)\in{\rm spt}(f^{\mathrm{in}})} |v|,\sup\limits_{(x,v)\in{\rm spt}(g^{\mathrm{in}})} |v|\right\}<\infty.
\end{equation*}
Then, there exist global unique nonnegative weak solutions $f,g\in L^{\infty}([0, \infty), L^1(\mathbb{R}^{2d}))$ to \eqref{A-1} with initial data $f^{\mathrm{in}}$ and $g^{\mathrm{in}}$, respectively, with the following properties:
\vspace{0.1cm}
\begin{enumerate}
	\item 
	(Propagation of compact supports in velocity variable): 
	\begin{equation}\label{eqt4.40}
		\sup\limits_{(x,v)\in{\rm spt}(f(t))} |v|\le P_\infty ,  \quad  \sup\limits_{(x,v)\in{\rm spt}(g(t))} |v|\le P_\infty, \quad t \in [0, \infty).
	\end{equation} 
	\item	
	(Conservation of total mass and its average velocity):
	\[ \frac{\di}{\di t}\int_{\mathbb{R}^{2d}}f(t, z) \di z=0, \quad  \frac{\di}{\di t}\int_{\mathbb{R}^{2d}}vf(t, z) \di z =0, \quad t \in [0, \infty). 
	\]
	\item
	(Finite-in-time stability in $W_p$):~ there exists a positive constant $C(\tau) = C(P_\infty, \beta, p, \tau)$ such that
	\begin{equation*}
		W_p(f(t),g(t))\le C(\tau) W_p(f^{\mathrm{in}},g^{\mathrm{in}}), \quad \forall~t\in[0, \tau),
	\end{equation*}
	where $\beta$ is the spatial decay of the communication weight $\phi$ (see \eqref{A-0-1}).
	\vspace{0.1cm}
\item	
(Weak flocking dynamics):~for $\beta\in [0,1)$, if the initial probability density function $f^{\mathrm{in}}$ of the initial measure $\mu^{\mathrm{in}}(\di x, \di v) = f^{\mathrm{in}} \di x \di v$ has an exponential decay $(\alpha > 0)$ or a polynomial decay with ($D>\frac{7}{1-\beta}$) in space variable $x$, and the coupling strength $\kappa$ is sufficiently large, then weak flocking \eqref{A-5} emerges asymptotically. 
\end{enumerate}	
\end{proposition}
\vspace{0.2cm}

\noindent Next,  we state the existence of a weak solution to \eqref{A-1}.
\begin{proposition}  \label{P2.2}
	\emph{\cite{kinetic3}} \label{pro:2.4}
	Suppose that the initial probability density function \( f^{\mathrm{in}} \) satisfies
	\[
	f^{\mathrm{in}} \in (L^1 \cap L^{\infty}_+)(\mathbb{R}^{2d}), \quad (|x|^2 + |v|^2)f^{\mathrm{in}} \in L^1(\mathbb{R}^{2d}).
	\]
	Then, there exists a weak solution $f \in {\mathcal C}([0, \tau); L^1(\mathbb{R}^{2d})) \cap L_+^\infty([0, \tau) \times \mathbb{R}^{2d})$  to \eqref{A-1} such that
	\[
	( |x|^2 + |v|^2)f(t) \in L^1(\mathbb{R}^{2d}), \quad 	\int_{\mathbb{R}^{2d}}f(t, z) \di z =\int_{\mathbb{R}^{2d}}f^{\mathrm{in}}(z) \di z, \quad t \in (0, \tau).
	\]
\end{proposition}
\begin{remark}\label{R2.3} 
	Below, we comment on the results of Proposition \ref{P2.2}. 
\begin{enumerate}
\item
Thanks to Proposition \ref{pro:2.4}, without loss of generality, we may assume 
	\[ \|f(t)\|_{1}\equiv1 \quad \forall~t \in [0, \tau). \]
\item	
By the same argument as in the proof of Proposition \ref{P2.2}, if $f^{\mathrm{in}}$ satisfies
	\[
	f^{\mathrm{in}} \in (L^1 \cap L_+^\infty)(\mathbb{R}^{2d}), \quad (|x|^D + |v|^D)f^{\mathrm{in}} \in L^1(\mathbb{R}^{2d}),
	\]
		then, there exists a weak solution $f \in {\mathcal C}([0, \tau); L^1(\mathbb{R}^{2d})) \cap L_+^\infty([0, \tau) \times \mathbb{R}^{2d})$  to \eqref{A-1} such that
	\[
	( |x|^D + |v|^D)f(t) \in L^1(\mathbb{R}^{2d}), \quad 	\int_{\mathbb{R}^{2d}}f(t, z) \di z =\int_{\mathbb{R}^{2d}}f^{\mathrm{in}}(z) \di z, \quad t \in (0, \tau).
	\]
\end{enumerate}	
\end{remark}
\subsection{Description of main results} \label{sec:2.3}
In this subsection, we discuss our main results on the uniqueness of weak solution and weak flocking behaviors of weak solutions. \newline

The first main result is concerned with the uniqueness property of a weak solution to \eqref{A-1} in which existence is already guaranteed in Proposition \ref{P2.2}.
\begin{theorem} \label{T2.1}
\emph{(Uniqueness of weak solution)}
Suppose that the initial datum \( f^{\mathrm{in}} \) satisfies
\[
f^{\mathrm{in}} \in (L^1 \cap L_+^{\infty})(\mathbb{R}^{2d}) \quad \mbox{and} \quad  (|x|^2 + {e^{\alpha|v|}})f^{\mathrm{in}} \in L^1(\mathbb{R}^{2d}), \quad \alpha>0.
\]
Then, a weak solution to \eqref{A-1}  is unique.
\end{theorem}
\begin{proof}  Since the proof is very lengthy, we postpone the proof in Section \ref{sec:3.2}.
\end{proof}
\begin{remark}
Different from Chen and Yin \cite{k6}, who used an energy method to obtain the uniqueness of the solution when $\beta>1$, we mainly use characteristic flow estimation and uniform boundedness of velocity moments.
\end{remark}
Next, the second result deals with the derivation of weak flocking properties. For this, we introduce two types of probability distribution functions with decay conditions at infinity in phase space.

\begin{definition} \label{D2.4}
Let $h = h(x, v)$ be a nonnegative and integrable function on $\mathbb{R}^{2d}$. 
\begin{enumerate}
	\item
	$h$ belongs to a class of exponentially decaying distributions in phase space, if there exist $\alpha > 0$ and $\delta\ge1$ such that 
	\[
	\int_{\mathbb{R}^{2d}} e^{\alpha(|x|+|v|^{\delta})} h(z) \di z < \infty.
	\]
	\vspace{0.2cm}
	\item
	$h$ belongs to a class of polynomially decaying distributions in phase space, if there exist $D_1, D_2 > 0$ such that 
	\[
	\int_{\mathbb{R}^{2d}} (|x|^{D_1}+ |v|^{D_2}) h(z) \di z < \infty.
	\]
\end{enumerate}
\end{definition}
\begin{remark} It is easy to see that if $h$ belongs to  a class of exponentially decaying distributions, then it also belongs to a class of polynomially decaying distributions.
\end{remark}

\begin{theorem}\label{T2.2}
\emph{(Weak flocking)} 
Let $f$ be a global weak solution to \eqref{A-1} with the initial datum $f^{\mathrm{in}}$. Then, the following assertions hold. 
\begin{enumerate}
\item
If $\beta,~\kappa,~\gamma,~ l_1>1, ~D_1,~ D_2$ and $f^{\mathrm{in}}$ satisfy 
{\begin{align}
\begin{aligned} \label{NN-1}
& 0 \leq \beta < 1, \quad \kappa> 0, \quad  \gamma > 1, \quad  \gamma\beta<1,\quad l_1>1, \quad   D_1 > \frac{2l_1}{(l_1-1)\left(\gamma -1\right)},    \\
&  D_2\ge \max\left\{D_1,~2l_1 \right\},\quad {\mathcal M}_p(f^{\mathrm{in}}, D_1, D_2) :=  \int_{\mathbb{R}^{2d}} (|x|^{D_1}+ |v|^{D_2}) f^{\mathrm{in}}(z) \di z < \infty, 
\end{aligned}
\end{align}}
then weak flocking \eqref{A-5} emerges at least algebraically fast:
\[
\int_{\mathbb{R}^{2d}} |V(t)- v^{\mathrm{in}}_c|^2f^{\mathrm{in}}(z) \di z \leq  {\mathcal O}(1)  (1+ t)^{-\big(\frac{(l_1-1)D_1}{l_1}(\gamma-1)\big)} \quad \mbox{as $t \to \infty$}. 
\]
Here, $\gamma$ is a free scale parameter of time-varying effective region (see \eqref{NewD-4-4} in Section \ref{sec:4.1}).
\vspace{0.2cm}
\item
If $\beta, ~\kappa,~\alpha$ and $f^{\mathrm{in}}$ satisfy
\begin{align} \label{NN-11}  
	0 \leq \beta < 1, \quad \kappa >0,  \quad  \alpha > 0, \quad  {\mathcal M}_e(f^{\mathrm{in}}, \alpha) :=  \int_{\mathbb{R}^{2d}} e^{\alpha \left(| x|+| v |\right)} f^{\mathrm{in}}(z) \di z  < \infty,\end{align}
then weak flocking \eqref{A-5} emerges with super-polynomial decay asymptotically; namely, the convergence is faster than any algebraic rate. Furthermore, if we additionally require that system parameters and initial datum satisfy
{\begin{align}\label{NN-111}
		 \delta>1, \quad \beta\in\Bigg[0,\frac{\delta-1}{\delta}\Bigg), \quad \kappa > 0, \quad \alpha > 0, \quad  {\mathcal M}_e(f^{\mathrm{in}}, \alpha, \delta) := \int_{\mathbb{R}^{2d}} e^{\alpha \left(| x|+| v |^{\delta}\right)} f^{\mathrm{in}}(z) \di z < \infty,\end{align}}
then weak flocking \eqref{A-5} emerges at least exponentially fast:
{\[
\int_{\mathbb{R}^{2d}} |V(t)- v^{\mathrm{in}}_c|^2f^{\mathrm{in}}(z) \di z \leq  {\mathcal O}(1) e^{-|\mathcal{O}(1)|  (1+t )^{1-\beta\frac{\delta}{\delta-1}}} \quad \mbox{as $t \to \infty$}. 
\]}

\vspace{0.1cm}
\item
If $\beta,~\kappa$ and $f^{\mathrm{in}}$ satisfy
\begin{align}\label{NN-1111}
	\beta = 1, \quad {\mathcal M}_e(f^{\mathrm{in}}, \alpha) :=  \int_{\mathbb{R}^{2d}} e^{\alpha | x|} f^{\mathrm{in}}(z) \di z  < \infty,\quad \sup\limits_{(x,v)\in{\rm spt}(f^{\mathrm{in}})} |v|\le P_\infty,  \quad \frac{ \kappa}{\left(\frac{1}{\alpha}+1\right)P_\infty}>2,\end{align}
then weak flocking \eqref{A-5} still emerges at least algebraically fast:
{\[
	\int_{\mathbb{R}^{2d}} |V(t)- v^{\mathrm{in}}_c|^2f^{\mathrm{in}}(z) \di z \leq  {\mathcal O}(1) (1+t)^{-\frac{\kappa}{\left(\frac{1}{\alpha}+1\right)P_\infty}} \quad \mbox{as $t \to \infty$}. 
	\]}
\end{enumerate}
\end{theorem}
\begin{proof} 
Since the proofs for the above assertions are very lengthy, we leave their proofs in Section \ref{sec:4}. However, for readers' convenience, we briefly outline a proof strategy for the first assertion in several steps. \newline

Suppose that system parameters and initial datum satisfy \eqref{NN-1}. 
\begin{itemize}
\item
Step A:~we show  that the growth of $D_{1}$-th spatial moment is at most $\mathcal{O}\left((1+t)^{D_{1}}\right).$ See Lemma \ref{L3.1}.
\vspace{0.1cm}
\item
Step B:~we show that the mass over the monotonically decreasing set $ \Omega(t) := (B_{R_x(t)})^c \times \bbr^d$ with $R_x(t):=  C_1(1 + t)^{\gamma}$, 
has an algebraic decay with $\mathcal{O}(1) (1+t)^{-D_1(\gamma-1)}.$  See Corollary \ref{C3.1}.
\vspace{0.1cm}
\item
Step C:~we use the relation (iv) in Lemma \ref{L2.2}
to derive the Grönwall type differential inequality:
\begin{align*}
\begin{aligned}
& \dfrac{\di}{\di t} \int_{\mathbb{R}^{2d}} |V(t)- v^{\mathrm{in}}_c|^2f^{\mathrm{in}}(z) \di z  \\
& \hspace{0.5cm} \leq - \kappa|\mathcal{O}(1)| (1+t)^{-\beta\gamma} \int_{\mathbb{R}^{2d}} |V(t)- v^{\mathrm{in}}_c|^2f^{\mathrm{in}}(z) \di z +\mathcal{O}(1) \kappa (1+t )^{-\big(\frac{(l_1-1)D_1}{l_1}(\gamma-1)+ \beta\gamma\big)}.
		\end{aligned}
	\end{align*}
See Section \ref{sec:4.1} for details.	
\vspace{0.2cm}
\item
Step D:~we use a variant of the Grönwall type lemma to derive weak velocity alignment:
\[  \int_{\mathbb{R}^{2d}} |V(t)- v^{\mathrm{in}}_c|^2f^{\mathrm{in}}(z) \di z \lesssim  {\mathcal O}(1)  (1+ t)^{-\big(\frac{(l_1-1)D_1}{l_1}(\gamma-1)\big)}.     \]
\vspace{0.1cm}
\item
Step E:~ we use \eqref{NN-1} and the relation
\[ 	\quad\dfrac{\di}{\di t}\left(\int_{\mathbb{R}^{2d}} |X(t)-v^{\mathrm{in}}_ct- x^{\mathrm{in}}_c|^2f^{\mathrm{in}}(z) \di z \right)^{\frac{1}{2}}\le\Big(\int_{\mathbb{R}^{2d}} |V(t)- v^{\mathrm{in}}_c|^2f^{\mathrm{in}}(z) \di z \Big)^{\frac{1}{2}}  \]
to derive the weak spatial cohesion:
\[
\int_{\mathbb{R}^{2d}} |X(t)-v^{\mathrm{in}}_ct- x^{\mathrm{in}}_c|^2f^{\mathrm{in}}(z) \di z \leq {\mathcal O}(1).
\]
\end{itemize}
\end{proof}
\begin{remark}
We discuss several comments on the framework and the results of Theorem \ref{T2.1} and Theorem \ref{T2.2} as follows.
\begin{enumerate}
\item
Many probability distributions belong to an exponentially decaying distribution class, including well-known distributions such as Gaussian and sub-Gaussian distributions. Conversely, to fulfill the conditions for a polynomially decaying distribution, the high-order moment must be finite.
\vspace{0.1cm}
\item
The results in Theorem \ref{T2.1} and Theorem \ref{T2.2} demonstrate that as long as initial velocity distribution is sufficiently concentrated, the solution will be unique in the weak sense of definition \ref{D2.1}, and a certain degree of concentration will persist under the evolution of the KCS flow.
\vspace{0.1cm}
\item
Compared to the previous work \cite{w6}, our results extend their findings by relaxing the assumption of compact velocity support to allow for non-compact velocity support. In particular, we do not impose any condition on coupling strength $\kappa$ and improve the critical exponent. More precisely, we can impose a stronger assumption than \eqref{NN-1} by requiring that $D=\min\{D_1,D_2\}$ satisfies
\[D > \inf\limits_{l_1>1}\left(\max \Big \{\frac{2l_1}{(l_1-1)(\gamma -1)}, ~2l_1 \Big \}\right).\]
{Then, if we take $\gamma=\frac{1-\varepsilon}{\beta} ~(0<\varepsilon<1-\beta)$, we only need to require}
 \[{D>\inf\limits_{l_1>1}\left(\max\left\{\frac{2l_1}{(l_1-1)}\frac{\beta}{1-\varepsilon-\beta},\quad 2l_1\right\}\right).}\] This significantly improves the exponent $D>\frac{7}{1-\beta}$ in \cite{w6}, even though we deal with a fully non-compact setting.

\vspace{0.1cm}
\item
If we further require the velocity support of the solution to be bounded as in \cite{w6},  we can also provide the weak flocking behavior of KCS model when $\beta= 1$ for an exponentially decaying initial distribution. However, for a polynomial decay initial distribution, we can only deal with $\beta< 1,$ even if we assume the velocity support of the solution is bounded.
\vspace{0.1cm}
\item Our conclusions about the weak flocking behavior do not rely on the argument of characteristics flow. We can also obtain it by directly studying the solutions of the PDEs without characteristics flow. Moreover, we can also use a similar argument in our paper to deal with a large number of CS-type models with a Brownian component.  More precisely, if the diffusion coefficient satisfies an appropriate time-decaying condition, we can obtain the weak flocking behavior of the Cucker–Smale–Fokker–Planck type equation proposed in \cite{ks1} without requiring the kernel function to have a positive lower bound. 
\vspace{0.1cm}
\item 
{The decay estimates in Theorem~\ref{T2.2} reflect the tail behavior of initial data: exponential decay of initial data yields an exponential weak flocking estimate, whereas polynomial decay yields a polynomial one. However, we do not claim that these rates are optimal. Our proof is based on estimating the maximal growth of the spatial moment, and such an argument is generally not expected to provide a sharp convergence rate. Determining the optimal weak flocking rate for given initial data and system parameters remains an interesting open problem.}
\end{enumerate}
\end{remark}
\vspace{0.2cm}
In the following two sections, we provide proofs for Theorem \ref{T2.1} and Theorem \ref{T2.2}.

\section{Uniqueness of weak solution}\label{sec:3}
\setcounter{equation}{0}
In this section, we present the uniqueness of weak solutions to \eqref{A-1} using spatial and velocity moments estimates along the particle trajectories. 
\subsection{Propagation of spatial and velocity moments}\label{sec:3.1}
In this subsection, we study the propagation of spatial and velocity moments, and derive estimates on the mass of particles with sufficiently large velocities. 
\begin{lemma}\label{L3.1}
For {$D\ge 2$} and $\tau \in (0, \infty]$, let $f \in {\mathcal C}([0, \tau); L^1(\mathbb{R}^{2d})) \cap L_+^\infty([0, \tau) \times \mathbb{R}^{2d})$ be a weak solution to \eqref{A-1} such that 
\[
{\mathcal M}_p(f^{\mathrm{in}}, D) :=  \int_{\mathbb{R}^{2d}} (|x|^D+ |v|^D) f^{\mathrm{in}}(z) \di z < \infty.
\]
Then, we have the following estimates: for $t \in [0, \tau)$, 
	\begin{align*}
	\begin{aligned} \label{C-1-101}
& (i)~\int_{\mathbb{R}^{2d}}v f(t, z) \di z =	\int_{\mathbb{R}^{2d}}vf^{\mathrm{in}}(z) \di z. \\
& (ii)~ \int_{\mathbb{R}^{2d}} xf(t, z) \di z  =	\int_{\mathbb{R}^{2d}}x f^{\mathrm{in}}(z) \di z  + t \int_{\mathbb{R}^{2d}}vf^{\mathrm{in}}(z) \di z. \\
& (iii)~ \int_{\mathbb{R}^{2d}} | v|^{D}f(t, z) \di z \le	\int_{\mathbb{R}^{2d}} | v|^{D}f^{\mathrm{in}}(z) \di z. \\
& (iv)~\int_{\mathbb{R}^{2d}} | x|^{D}f(t, z) \di z   \le	\left[ \left(\int_{\mathbb{R}^{2d}} | x|^{D}f^{\mathrm{in}}(z) \di z \right)^{\frac{1}{D}}+ 
\left(\int_{\mathbb{R}^{2d}} | v|^{D}f^{\mathrm{in}}(z) \di z \right)^{\frac{1}{D}}t\right ]^{D}.
	\end{aligned}
	\end{align*}

\end{lemma}
\begin{proof}
(i)~It follows from the weak formulation in Definition \ref{D2.1} that for any $\psi \in {\mathcal C}_c^1(\mathbb{R}^{2d})$, 
	\begin{equation}\label{424}
		\dfrac{\di}{\di t}\int_{\mathbb{R}^{2d}}\psi f(t, z) \di z =\int_{\mathbb{R}^{2d}} \Big( v \cdot\nabla_{ x}\psi +\nabla_{ v}\psi \cdot L[f] \Big) f(t, z) \di z.
	\end{equation}
	Since the map $ v~\mapsto~| v|^D$ is not an admissible test
	function in \eqref{424}, we introduce a smooth cut-off function \(\varphi_R\) such that  
	\[
	\varphi_R( v) = 
	\begin{cases} 
		1, & \mbox{for }~ v\in B_R, \\ 
		0, & \mbox{for }~ v\in \mathbb{R}^d \setminus B_{2R},
	\end{cases} 
	\quad \mbox{and} \quad  |\nabla \varphi_R| \leq \frac{2}{R},
	\]
	where $B_r = B_r(0)$ is the ball with a radius $r$ centered at the origin. Then, we can use the definition of a weak solution with the function to find 
		\begin{align}
		\begin{aligned} \label{C-0-110}
		 \dfrac{\di}{\di s}\int_{\mathbb{R}^{2d}}v f(s, z) \di z =-\kappa \int_{\mathbb{R}^{4d}}\phi( |x_{\star}-x|)(v-v_{\star})f(s, z_{\star})f(s, z) \di z \di z_\star =0.
	\end{aligned}
	\end{align}
	We integrate \eqref{C-0-110} from $s= 0$ to $s= t$ to find the desired total momentum.
	
	\vspace{0.2cm}
	
\noindent (ii)~We use a similar argument as in (i) and the result of (i) to get 
\begin{equation}\label{New-0}
			\dfrac{\di}{\di s}\int_{\mathbb{R}^{2d}} x f(s, z) \di z = \int_{\mathbb{R}^{2d}} v  f(s, z) \di z =  \int_{\mathbb{R}^{2d}} v  f^{\mathrm{in}}(z) \di z.
	\end{equation}	
Again, we integrate \eqref{New-0} to derive the desired estimate. 

\vspace{0.2cm}	
	
\noindent (iii)~Recall that 
	\begin{equation} \label{New-1}
	 \partial_t f = -\nabla_x \cdot (v f) - \nabla_v \cdot (L[f] f). 
	\end{equation}
	
\noindent We use the exchange map $(x, v)~~\Longleftrightarrow~~(x_\star, v_{\star})$, the symmetry of $\phi(\cdot)$ under the exchange map,  \eqref{New-1} and 
\[ \nabla_v ( | v |^D)  = D | v |^{D-2} v \]
 to find 
		\begin{align}\label{C-0-10}
		\begin{aligned}
			&\dfrac{\di}{\di t}\int_{\mathbb{R}^{2d}} | v |^D f(t, z) \di z = \int_{\mathbb{R}^{2d}} | v |^D \partial_t f(t, z) \di z \\
			&=-\kappa D\int_{\mathbb{R}^{4d}}\phi(|x_{\star}-x|) | v |^{D-2}\langle v,v-v_{\star}\rangle f(t, z_{\star})f(t, z) \di z\di z_\star \\	
			&=-\frac{\kappa D}{2}\int_{\mathbb{R}^{4d}}\phi(|x_{\star}-x|) \langle | v|^{D-2} v-| v_{\star}|^{D-2} v_{\star},v-v_{\star}\rangle f(t, z_\star)   f(t,z) \di z_{\star}  \di z  \\
					&\le-\frac{\kappa D}{2}\int_{\mathbb{R}^{4d}}\phi(|x_{\star}-x|)\left(| v|^{D}+ | v_{\star} |^{D}-| v_{\star} |^{D-1} | v|-| v|^{D-1} | v_{\star}|\right)f(t, z_{\star})f(t, z) \di z_\star  \di z  \\		
			&=-\frac{\kappa D}{2}\int_{\mathbb{R}^{4d}}\phi(|x_{\star}-x|)(| v|^{D-1}-| v_{\star}|^{D-1})\left(| v|-| v_{\star}|\right)f(t, z_{\star})f(t, z) \di z_\star  \di z \\
			&\le 0,
		\end{aligned}
	\end{align}
	where $\langle \cdot, \cdot \rangle$ is the standard inner product in $\bbr^d$. This yields the desired third estimate.
	
	\vspace{0.2cm}
	
\noindent (iv)~We can use the same technique as in (i) and the H\"{o}lder inequality to get 
		\begin{align}\label{C-0-11}
		\begin{aligned}
		& \dfrac{\di}{\di t}\int_{\mathbb{R}^{2d}} | x|^D f(t, z)  \di z \\
		& \hspace{0.5cm} =D\int_{\mathbb{R}^{2d}} | x|^{D-2}\langle x,v \rangle f(t, z) \di z \le D\int_{\mathbb{R}^{2d}} | x|^{D-1} |v |f(t, z) \di z \\
		& \hspace{0.5cm} \le D\left(\int_{\mathbb{R}^{2d}} | x|^{D}f(t, z) \di z \right)^{\frac{D-1}{D}}\left(\int_{\mathbb{R}^{2d}} |v|^{D}f(t, z) \di z \right)^{\frac{1}{D}}.
		\end{aligned}
	\end{align}
	This implies 
		\begin{equation}\label{C-0-111}
		\dfrac{\di}{\di t}\left(\int_{\mathbb{R}^{2d}} | x|^{D}f(t, z) \di z \right)^{\frac{1}{D}}
			\le \left(\int_{\mathbb{R}^{2d}} |v|^{D}f(t, z) \di z \right)^{\frac{1}{D}} 
			\le \left(\int_{\mathbb{R}^{2d}} | v|^{D}f^{\mathrm{in}}(z) \di z \right)^{\frac{1}{D}}.
	\end{equation}
	Now, we integrate the above relation from $s= 0$ to $s= t$ to obtain
		\begin{align*}
	\int_{\mathbb{R}^{2d} } | x|^{D}f(t, z) \di z \le	\left[  \left(\int_{\mathbb{R}^{2d}} | x|^{D}f^{\mathrm{in}}(z) \di z \right)^{\frac{1}{D}}+\left(\int_{\mathbb{R}^{2d}} | v|^{D}f^{\mathrm{in}}(z) \di z \right)^{\frac{1}{D}}t\right]^{D }.
	\end{align*}

\end{proof}
\begin{remark} Below, we comment on the result of Lemma \ref{L3.1}.
\begin{enumerate}
\item
If we assume 
\[ x_c(0) =0, \quad v_c(0) =0, \]
then we have
\[
 x_c(t) =0, \quad v_c(t) =0, \quad t \geq 0.
\]
\item
It is easy to see that 
\[ 
\int_{\mathbb{R}^{2d}} | v|^{D}f(t, z) \di z \lesssim 1, \quad \int_{\mathbb{R}^{2d}} |x|^{D}f(t, z) \di z \lesssim (1 + t)^D.
\]
\item {Due to conclusion (iv) in Lemma \ref{L3.1}, we need to assume that \(D_2 \ge D_1\) in \eqref{NN-1}.}
\end{enumerate}
\end{remark}

As a direct application of Lemma \ref{L3.1}, we have estimates on spatial and velocity moments for large position and velocity. 
\begin{corollary} \label{C3.1}
Let $f$ be a weak solution to \eqref{A-1} with the initial datum $f^{\mathrm{in}}$ such that
\[
{\mathcal M}_p(f^{\mathrm{in}}, D) = \int_{\mathbb{R}^{2d}}( |x|^D+ |v|^D) f^{\mathrm{in}}(z) \di z < \infty,\quad D\ge2,
\]
and let $R_x(t)$ and $R_v(t)$ be nonnegative locally bounded functions. Then, we have
\begin{align}
\begin{aligned} \label{New-2}
& (i)~\int_{ \bbr_x^d \times B^c_{R_v(t)}} f(t, z) \di z \le\frac{\mathcal{M}_{p}(f^{\mathrm{in}},D)}{R_v(t)^D}. \\
& (ii)~\int_{B^c_{R_x(t)}  \times \bbr_v^d}   f(t, z) \di z \le \frac{2^{D-1} \mathcal{M}_{p}(f^{\mathrm{in}},D) (1 + t)^D}{R_x(t)^D}.
\end{aligned}
\end{align}
\end{corollary}
\begin{proof}
Let $R_x(t)$ and $R_v(t)$ be arbitrary time-dependent functions. Then, it follows from Lemma \ref{L3.1} that 
\begin{align*}
\begin{aligned}
& R_v(t)^D \int_{ \bbr_x^d \times B^c_{R_v(t)}} f(t,z) \di z   \le \int_{\mathbb{R}^{2d}} |v|^Df(t, z) \di z = \mathcal{M}_{p}(f^{\mathrm{in}}, D),\\
& R_x(t)^D   \int_{B^c_{R_x(t)}  \times \bbr_v^d}   f(t, z) \di z \leq     \int_{B^c_{R_x(t)}  \times \bbr_v^d}  |x|^D  f(t, z) \di z  \\
& \hspace{0.5cm}  \leq \int_{\bbr^{2d}} |x|^D f(t, z) \di z  \le \left[  \left(\int_{\mathbb{R}^{2d}} | x|^{D}f^{\mathrm{in}}(z) \di z \right)^{\frac{1}{D}}+ \left(\int_{\mathbb{R}^{2d}} | v|^{D}f^{\mathrm{in}}(z) \di z \right)^{\frac{1}{D}}t \right]^{D}  \\
& \hspace{0.5cm}  \le  2^{D-1} \Big[ \left(\int_{\mathbb{R}^{2d}} | x|^{D}f^{\mathrm{in}}(z) \di z \right) + \left(\int_{\mathbb{R}^{2d}} | v|^{D}f^{\mathrm{in}}(z) \di z \right) t^D \Big ] \\
& \hspace{0.5cm} \leq 2^{D-1} {\mathcal M}_p(f^{\mathrm{in}}, D) (1 + t)^D,
\end{aligned}
\end{align*}
where we used the inequality:
\[ (|a| + |b|)^{D} \leq 2^{D-1} (|a|^D + |b|^D),~~D \geq 1. \]
The above estimates yield the desired estimates. 
\end{proof}
\begin{remark} 
	If we choose a super-linearly growing function $R_x(t)$ such that 
\[ R_x(t) \geq C (1 + t)^{1+ \varepsilon} \quad \mbox{for some $\varepsilon > 0$}, \]
then the masses over the regions $\bbr_x^d \times B^c_{R_v(t)}$ and  $B^c_{R_x(t)}  \times \bbr_v^d$ in phase space:
\[ \int_{ \bbr_x^d \times B^c_{R_v(t)}} f(t, z) \di z  \quad \mbox{and} \quad \int_{B^c_{R_x(t)}  \times \bbr_v^d}   f(t, z) \di z  \]	
 decay to zero as $t \to \infty$. 
\end{remark}

Note that in Lemma \ref{L3.1}, we deal with the propagation of spatial-velocity moments along the kinetic flow. Next, we study estimates on the integrals of $f$ with exponential spatial and velocity weights.
\begin{lemma}\label{L3.2}
	For $\tau \in (0, \infty]$, let $f \in {\mathcal C}([0, \tau); L^1(\mathbb{R}^{2d})) \cap L_+^\infty([0, \tau) \times \mathbb{R}^{2d})$ be a weak solution to \eqref{A-1} such that 
	\[
	{\mathcal M}_{e}(f^{\mathrm{in}}, \alpha, \delta) :=  \int_{\mathbb{R}^{2d}} e^{\alpha(|x|+ |v|^{\delta})} f^{\mathrm{in}}(z) \di z < \infty,\quad \alpha>0, \quad \delta\ge1.
	\]
	Then, the following assertions hold.
\begin{enumerate}
\item
Propagation of exponentially weighted mass:	
\begin{equation} \label{NN-3}
\begin{cases}
\displaystyle \int_{\mathbb{R}^{2d}}e^{\alpha | v|}f(t, z) \di z \le	\int_{\mathbb{R}^{2d}} e^{\alpha | v|}f^{\mathrm{in}}(z) \di z, \vspace{6pt}\\
\displaystyle \int_{\mathbb{R}^{2d}}e^{\alpha | v|^{\delta}}f(t, z) \di z \le	\int_{\mathbb{R}^{2d}} e^{\alpha | v|^{\delta}}f^{\mathrm{in}}(z) \di z, \vspace{6pt}\\
\displaystyle \frac{\di }{\di t}\int_{\mathbb{R}^{2d}} e^{\alpha | x|}f(t, z) \di z   \leq	\alpha \int_{\mathbb{R}^{2d}} e^{\alpha | x|} |v| f(t, z) \di z.\vspace{6pt}\\
\end{cases}
\end{equation}
\item
If the projected velocity support of $f$ is uniformly bounded:
	\begin{equation}\label{eqt4.401}
		\sup\limits_{(x,v)\in{\rm spt}(f(t))} |v|\le P_\infty,
	\end{equation} 
then we have
\[
		\int_{\mathbb{R}^{2d}} e^{\alpha | x|}f(t, z) \di z\le e^{\alpha P_\infty t}\int_{\mathbb{R}^{2d}} e^{\alpha | x|}f^{\mathrm{in}}( z) \di z. \]
\end{enumerate}		
\end{lemma}
\begin{proof}
(1) The derivations for the first and third estimates in \eqref{NN-3} are similar to Lemma \ref{L3.1},  hence we omit them here. We only provide proof of the second estimate here. For $\Psi(v):=e^{\alpha |v|^\delta}$ with $\alpha>0$ and $\delta\ge 1$, we formally compute
\begin{align*}
	\frac{\di }{\di t}\int_{\mathbb R^{2d}} \Psi(v) f(t,z) \di z
	&= \int_{\mathbb R^{2d}} \Psi(v)\,\partial_t f(t,z) \di z \\
	&= \int_{\mathbb R^{2d}} \nabla_v \Psi(v)\cdot L[f](t,x,v)\, f(t,z) \di z \\
	&= -\kappa \int_{\mathbb R^{4d}} \phi(|x-x_{\star}|)\,\nabla\Psi(v)\cdot (v-v_{\star})\, f(t,z)f(t,z_{\star})\,\di z\di z_{\star} \\
	&= -\frac{\kappa}{2}\int_{\mathbb R^{4d}} \phi(|x-x_{\star}|)
	\big(\nabla\Psi(v)-\nabla\Psi(v_{\star})\big)\cdot (v-v_{\star})\, f(t,z)f(t,z_{\star})\,\di z\di z_{\star}.
\end{align*}
Since $\Psi(v)=e^{\alpha |v|^\delta}$ is convex for $\delta\ge 1$, its gradient is monotone, i.e.,
\[
\big(\nabla\Psi(v)-\nabla\Psi(v_{\star})\big)\cdot (v-v_{\star})\ge 0.
\]
Therefore, we have
\[
\frac{\di }{\di t}\int_{\mathbb R^{2d}} e^{\alpha |v|^\delta} f(t,z)\di z \le 0,
\]
which yields
\[
\int_{\mathbb R^{2d}} e^{\alpha |v|^\delta}f(t,z) \di z
\le
\int_{\mathbb R^{2d}} e^{\alpha |v|^\delta}f^{\mathrm{in}}(z) \di z.
\]\newline

\noindent (2) We use the third estimate in \eqref{NN-3} to see that 
\[
\frac{\di }{\di t}\int_{\mathbb{R}^{2d}} e^{\alpha | x|}f(t, z) \di z   \leq	\alpha \int_{\mathbb{R}^{2d}} e^{\alpha | x|} |v| f(t, z) \di z \leq \alpha P_{\infty} \int_{\mathbb{R}^{2d}} e^{\alpha | x|} f(t, z) \di z.
\]
This yields the desired estimate.
\end{proof}

As a direct application of Lemma \ref{L3.2}, we can estimate the mass outside of the effective region.
\begin{corollary}\label{C3.2}
	Let $R_x(t)$ and $R_v(t)$ be nonnegative locally bounded functions. 
	\begin{enumerate}
		\item Suppose that the initial datum $f^{\mathrm{in}} \in L^1(\mathbb{R}^{2d})$ satisfy 
		\begin{equation*}
			P_\infty :=\sup\limits_{(x,v)\in{\rm spt}(f^{\mathrm{in}})} |v|<\infty, \qquad  \int_{\mathbb{R}^{2d}} e^{\alpha|x|} f^{\mathrm{in}}(z) \di z \leq{\mathcal M}_{e}(f^{\mathrm{in}},\alpha) < \infty,
		\end{equation*}
		and  let $f$ be a global weak solution to \eqref{A-1}. Then we have
		\begin{equation} \label{New-2-1}
			\int_{B_{R_x(t)}^c  \times \bbr^d}   f(t, z) \di z \le \mathcal{M}_{e}(f^{\mathrm{in}}, \alpha) e^{\alpha P_\infty t-\alpha R_x(t)}.
		\end{equation}
			\item Suppose that the initial datum $f^{\mathrm{in}} \in L^1(\mathbb{R}^{2d})$ satisfy 
		\[
		{\mathcal M}_{e}(f^{\mathrm{in}}, \alpha, \delta) :=  \int_{\mathbb{R}^{2d}} e^{\alpha(|x|+ |v|^{\delta})} f^{\mathrm{in}}(z) \di z < \infty,\quad \alpha>0, \quad \delta\ge1.
		\]
			 Then we have
		\begin{equation} \label{New-2-1-1}
			\int_{ \bbr^d\times B_{R_x(t)}^c  }   f(t, z) \di z \le \mathcal{M}_{e}(f^{\mathrm{in}}, \alpha, \delta) e^{-\alpha \left(R_v(t)\right)^{\delta}}.
		\end{equation}
	\end{enumerate}

\end{corollary}
\begin{proof} 
\noindent(1) Note that 
\begin{align}
\begin{aligned} \label{New-2-2} 
&  e^{\alpha R_x(t)} \int_{B_{R_x(t)}^c \times \bbr^d} f(t, z) \di z   \leq 	\int_{B_{R_x(t)}^c \times \bbr^d} e^{\alpha | x|}f(t, z) \di z \\
& \hspace{2cm} \leq  \int_{\mathbb{R}^{2d}} e^{\alpha | x|}f(t, z) \di z  \leq e^{\alpha P_\infty t}\int_{\mathbb{R}^{2d}} e^{\alpha | x|}f^{\mathrm{in}}( z) \di z, \\
\end{aligned}
\end{align}
where we used (2) of Lemma \ref{L3.2} in the last inequality.  \newline

On the other hand, we have
\begin{equation} \label{New-2-3}
\int_{\mathbb{R}^{2d}} e^{\alpha | x|}f^{\mathrm{in}}( z) \di z \leq \int_{\mathbb{R}^{2d}} e^{\alpha (| x| + |v|) }f^{\mathrm{in}}( z) \di z = M_e(f^{\mathrm{in}}, \alpha) < \infty.
\end{equation}
Now, we combine \eqref{New-2-2} and \eqref{New-2-3} to get 
\[
 e^{\alpha R_x(t)} \int_{B_{R_x(t)}^c \times \bbr^d} f(t, z) \di z  \leq e^{\alpha P_\infty t}\int_{\mathbb{R}^{2d}} e^{\alpha | x|}f^{\mathrm{in}}( z) \di z \leq  e^{\alpha P_\infty t} M_e(f^{\mathrm{in}}, \alpha). 
 \]
This yields the desired estimates.\newline 

\noindent(2) The derivation for  \eqref{New-2-1-1} is similar to Corollary \ref{C3.1} (ii),  hence we omit it here.
\end{proof}
In the following lemma, we provide the time-derivative of the second velocity moment around the average velocity.
\begin{lemma}\label{L3.3}
	For $\tau \in (0, \infty]$, let $f \in {\mathcal C}([0, \tau); L^1(\mathbb{R}^{2d})) \cap L_+^\infty([0, \tau) \times \mathbb{R}^{2d})$ be a weak solution to \eqref{A-1} with initial datum $f^{\mathrm{in}}$ satisfying 
	\[
	\int_{\bbr^{2d}} (|x|^2 + |v|^2) f^{\mathrm{in}}(z) \di z < \infty. 
	\]
	Then the kinetic energy is non-increasing:
	\[
	\frac{\di}{\di t} \int_{{\mathbb R}^{2d}} |v - v_c(t)|^2 f(t, z) \di z =-\kappa \int_{{\mathbb R}^{4d}} \phi(|x-x_\star|) |v-v_\star|^2  f(t, z_\star) f(t, z) \di z_\star  \di z\le0.
	\]
\end{lemma}
\begin{proof} It follows from Lemma \ref{L3.1} that 
	\begin{equation} \label{New-2-4}
	v_c(t) = v_c^{\mathrm{in}}, \quad t \geq 0. 
	\end{equation}
	We use \eqref{New-2-4} and $ \partial_t f = -\nabla_x \cdot (v f) - \nabla_v \cdot (L[f] f)$ to find 
	\begin{align*}
		\begin{aligned}
			&\frac{\di}{\di t} \int_{{\mathbb R}^{2d}} |v - v_c(t)|^2 f(t, z) \di z  =\frac{\di}{\di t} \int_{{\mathbb R}^{2d}} |v - v_c^{\mathrm{in}}|^2 f(t, z) \di z  \\
			& \hspace{0.5cm} = \int_{{\mathbb R}^{2d}}  |v - v_c^{\mathrm{in}}|^2 \partial_t f(t, z) \di z   =  -  \int_{{\mathbb R}^{2d}}   |v - v_c^{\mathrm{in}}|^2  \nabla_v \cdot (L[f] f) \di z  \\
			& \hspace{0.5cm}  = -2 \kappa \int_{{\mathbb R}^{4d}} \phi(|x-x_{\star}|) \langle v - v_c^{\mathrm{in}}, v-v_{\star} \rangle  f(t,  z_\star) f(t, z) \di z_\star \di z   \\
			& \hspace{0.5cm}   =2 \kappa \int_{{\mathbb R}^{4d}} \phi(|x-x_{\star}|) \langle v_{\star} - v_c^{\mathrm{in}}, v-v_{\star} \rangle  f(t,  z_\star) f(t, z) \di z_\star \di z   \\
			& \hspace{0.5cm}  =- \kappa \int_{{\mathbb R}^{4d}} \phi(|x-x_{\star}|) |v-v_{\star}|^2  f(t,  z_\star) f(t, z) \di z_\star \di z.
		\end{aligned}
	\end{align*}
\end{proof}		
With the above preparations, we can show the exponential emergence of weak flocking under special case $\beta=0$.
\begin{proposition} \label{P3.1}
Suppose that decay of communication weight and initial datum satisfy 
			\[ \beta = 0, \quad \kappa > 0, \quad  (|x|^2 + |v|^2)  f^{\mathrm{in}} \in L^1(\mathbb{R}^{2d}), \quad v_c(0) = 0, \]
			and let $f$ be a global weak solution to \eqref{A-1}.  Then we have 
			\begin{align*}\label{eq:300}
				\int_{\mathbb{R}^{2d}} |v|^2f(t, z) \di z =e^{-2 \kappa t}	\int_{\mathbb{R}^{2d}} |v|^2f^{\mathrm{in}}(z) \di z,  \quad  \forall~ t\ge0.
			\end{align*}
		\end{proposition}
		\begin{proof} By the assumptions, we have
		\[ \phi \equiv 1 \quad v_c(t) = v_c(0) = 0, \quad t > 0. \]
		Then, we use Lemma \ref{L3.2} to find 
\[ \frac{\di}{\di t} \int_{{\mathbb R}^{2d}} |v|^2 f(t, z) \di z = -\kappa \int_{{\mathbb R}^{4d}} |v - v_{\star}|^2 f(t, z_\star) f(t, z)  \di z_\star  \di z = -2 \kappa \int_{{\mathbb R}^{2d}} |v |^2 f(t, z) \di z. \]
			This yields an exponential decay estimate. 
		\end{proof}
\subsection{Proof of Theorem \ref{T2.1}}\label{sec:3.2}
Suppose the initial datum $f^{\mathrm{in}}$ satisfies positivity, boundedness and suitable decay conditions:
\[
f^{\mathrm{in}} \in (L^1 \cap L_+^\infty)(\mathbb{R}^{2d}), \quad (|x|^2 + e^{\alpha|v|})f^{\mathrm{in}} \in L^1(\mathbb{R}^{2d}), 
\]
and let $f$ be a global weak solution to \eqref{A-1} whose existence is guaranteed by Proposition \ref{P2.2}. Then, we use Lemma \ref{L3.1} and Lemma \ref{L3.2} to see 
				\begin{align}\label{vD}
				\int_{\mathbb{R}^{2d}} | v|^{D}f(t, z) \di z \le \int_{\mathbb{R}^{2d}} | v|^{D}f^{\mathrm{in}}(z) \di z, \quad 	\int_{\mathbb{R}^{2d}} e^{\alpha |v|}f(t, z) \di z \le \int_{\mathbb{R}^{2d}} e^{\alpha |v|}f^{\mathrm{in}}(z) \di z.
			\end{align}
Suppose that $g$ is another weak solution to \eqref{A-1} with the same initial datum $g^{\mathrm{in}} = f^{\mathrm{in}}$. We want to show that for $t \in [0, \tau)$, 
\[ f(t) \equiv g(t) \quad \mbox{a.e.~~in $\bbr^{2d}$}.  \]
To see this, we set  $(X_{f}(t),V_{f}(t))$ and $(X_{g}(t),V_{g}(t))$ to be the particle trajectories generated by $f$ and $g$, respectively:
\begin{equation} \label{New-D-1}
	\begin{cases}
		\displaystyle {\dot X}_f(t)=V_f (t), \quad t > 0,   \vspace{3pt}\\
		\displaystyle {\dot V}_f (t)= L[f](t, X_f(t), V_f(t)], \vspace{3pt}\\
		\displaystyle X_f(0)=x, \quad  V_f(0)=v, \vspace{3pt}\\
		\displaystyle f(t) =(X_{f}(t),V_{f}(t) )\# f^{\mathrm{in}},
	\end{cases}
\quad \mbox{and} \qquad  
	\begin{cases}
		\displaystyle {\dot X}_g(t)=V_g (t), \quad t > 0, \vspace{3pt}  \\
		\displaystyle {\dot V}_g (t)= L[g](t, X_g(t), V_g(t)], \vspace{3pt}\\
		\displaystyle X_g(0)=x, \quad  V_g(0)=v,\vspace{3pt} \\
		\displaystyle g(t) =(X_{g}(t),V_{g}(t) )\# f^{\mathrm{in}}.
	\end{cases}
\end{equation}
Next, we introduce a deviation functional $\Delta[f, g]$ between two particle trajectories which was first introduced in \cite{k6}:
\begin{equation*}
	\Delta[f,g](t):=\int_{\mathbb{R}^{2d}} \Big ( |X_{f}(t)-X_{g}(t)|+ |V_{f}(t)-V_{g}(t)| \Big )f^{\mathrm{in}}(z) \di z, \quad \forall~t\in[0,\tau].
\end{equation*}
Then, we claim that 
\begin{equation} \label{New-D-3}
\Delta[f,g](t) \equiv 0, \quad \forall~t \in [0, \tau].
\end{equation}
Once the claim \eqref{New-D-3} is verified, then we derive the uniqueness:
\[ f \equiv g. \]
Since 
\[ X_f(0)= X_g(0) \quad \mbox{and} \quad  V_f(0) = V_g(0), \]
it is easy to see 
\begin{equation} \label{New-D-4}
 \Delta[f,g](0) = 0.
 \end{equation}
{\it Proof of \eqref{New-D-3}}: It suffices to derive an integral form of  Grönwall's inequality for $\Delta[f, g]$. For this, we split the derivation into four steps. \newline

\noindent $\bullet$~Step A.1: We use $\eqref{New-D-1}_1$ and $\eqref{New-D-1}_2$ to get 
\begin{align}\label{4:23}
	\begin{aligned}
		|X_{f}(t)-X_{g}(t) | \le \int_0^t |V_{f}(s)-V_{g}(s) |\di s.
	\end{aligned}
\end{align}
This and Fubini's theorem yield
\begin{equation} \label{New-D-4-0}
\int_{\mathbb{R}^{2d}} |X_{f}(t)-X_{g}(t)| f^{\mathrm{in}}(z) \di z \leq  \int_0^t\int_{\mathbb{R}^{2d}} |V_{f}(s)-V_{g}(s)|  f^{\mathrm{in}}(z) \di z \di s.
\end{equation}

\vspace{0.2cm}

\noindent $\bullet$~Step A.2:  Recall that 
\begin{equation} \label{New-D-4-1}
L[f](t,x,v)=-\kappa \int_{\mathbb{R}^{2d}} \phi( |x-x_{\star}|) \left(v-v_{\star}\right)  f(t, z_{\star}) \di z_\star.
 \end{equation}
Again, we use $\eqref{New-D-1}_1$ and $\eqref{New-D-1}_2$ to obtain
\begin{equation} \label{New-D-5}
\frac{\di}{\di t} |V_f(t) - V_g(t)| \leq  \Big| L[f](t, X_f(t), V_f(t)) -  L[g](t, X_g(t), V_g(t))  \Big|.
 \end{equation}
We integrate \eqref{New-D-5} from $0$ to $t$, and then use $ (V_f - V_g)(0) = 0$, $\eqref{New-D-1}_4$,  to get
\begin{equation} \label{New-D-66}
| V_f(t) - V_g(t)| \leq \int_0^t \Big| L[f](s, X_f(s), V_f(s)) -  L[g](s, X_g(s), V_g(s)) \Big| \di s.
\end{equation}
For $s \leq t$, we set 
\begin{align}
\begin{aligned} \label{New-D-6}
& (x_{f,\star}(s),v_{f,\star}(s)) :=(X_{f}(s,0, x_{\star}, v_{\star}),V_{f}(s,0, x_{\star}, v_{\star})), \\
& (x_{g,\star}(s),v_{g,\star}(s)) :=(X_{g}(s,0, x_{\star}, v_{\star}),V_{g}(s,0, x_{\star}, v_{\star})).
\end{aligned}
\end{align}
Next, we use the Lipschitz continuity of $\phi$, \eqref{New-D-6} and 
\[ \| f^{\mathrm{in}} \|_1 = 1, \qquad   \frac{1}{(1+|X_{f}(s)-x_{f,\star}(s)|^2)^{\frac{\beta}{2}}} \leq 1 \]
 to find that for any $( x, v)\in{\rm spt}f^{\mathrm{in}}$
\begin{align} 
\begin{aligned} \label{423}
        & \Big| L[f](s, X_f(s), V_f(s)) -  L[g](s, X_g(s), V_g(s)) \Big|  \\
        & \hspace{0.2cm} = \kappa  \Big|  \int_{\mathbb{R}^{2d}} \phi( |X_f(s)-x_{\star}|) (V_f(s)-v_{\star} )  f(s, z_{\star}) \di z_\star \\
        & \hspace{1cm} - \int_{\mathbb{R}^{2d}} \phi( |X_g(s) -x_{\star}|) (V_g(s) - v_{\star})  {g}(s, z_{\star}) \di z_\star  \Big| \\	
	& \hspace{0.2cm}  \leq  \kappa \int_{\mathbb{R}^{2d}}\left |\frac{V_{f}(s)-v_{f,\star}(s)}{(1+|X_{f}(s)-x_{f,\star}(s)|^2)^{\frac{\beta}{2}}}-\frac{V_{g}(s)-v_{g,\star}(s)}{(1+|X_g(s)-x_{g,\star}(s)|^2)^{\frac{\beta}{2}}}\right |f^{\mathrm{in}}(z_{\star}) \di z_\star \\
	& \hspace{0.2cm}  \le  \kappa \int_{\mathbb{R}^{2d}}\left |\frac{V_{f}(s)-v_{f,\star}(s)}{(1+|X_{f}(s)-x_{f,\star}(s)|^2)^{\frac{\beta}{2}}}-\frac{V_{g}(s)-v_{g,\star}(s)}{(1+|X_{f}(s)-x_{f,\star}(s)|^2)^{\frac{\beta}{2}}}\right |f^{\mathrm{in}}(z_{\star}) \di z_\star  \\
	& \hspace{0.4cm} +  \kappa \int_{\mathbb{R}^{2d}}\left |\frac{V_{g}(s)-v_{g,\star}(s)}{(1+|X_{f}(s)-x_{f,\star}(s)|^2)^{\frac{\beta}{2}}}-\frac{V_{g}(s)-v_{g,\star}(s)}{(1+|X_{g}(s)-x_{g,\star}(s)|^2)^{\frac{\beta}{2}}}\right |f^{\mathrm{in}}(z_{\star}) \di z_\star  \\
	 & \hspace{0.2cm}   \le  \kappa \int_{\mathbb{R}^{2d}}  \left |\frac{V_{f}(s)-V_g(s)}{(1+|X_{f}(s)-x_{f,\star}(s)|^2)^{\frac{\beta}{2}}} \right |  +  \left |\frac{v_{f,\star}(s)-v_{g,\star}(s)}{(1+|X_{f}(s)-x_{f,\star}(s)|^2)^{\frac{\beta}{2}}} \right | f^{\mathrm{in}}(z_{\star}) \di z_\star  \\
	  & \hspace{0.4cm} +  \kappa \int_{\mathbb{R}^{2d}}\left |\frac{V_{g}(s)-v_{g,\star}(s)}{(1+|X_{f}(s)-x_{f,\star}(s)|^2)^{\frac{\beta}{2}}}-\frac{V_{g}(s)-v_{g,\star}(s)}{(1+|X_{g}(s)-x_{g,\star}(s)|^2)^{\frac{\beta}{2}}}\right |f^{\mathrm{in}}(z_{\star}) \di z_\star  \\
	  & \hspace{0.2cm}  \le  \kappa |V_{f}(s)-V_{g}(s) |+  \kappa \int_{\mathbb{R}^{2d}} |v_{f,\star}(s)-v_{g,\star}(s) |f^{\mathrm{in}}(z_{\star}) \di z_\star \\
	  & \hspace{0.4cm} +  \kappa C_{\beta} 
	\int_{\mathbb{R}^{2d}}\left( |v_{g,\star}(s)|+ |V_{g}(s) |\right) \left(|x_{f,\star}(s)-x_{g,\star}(s)|+ |X_{f}(s)-X_{g}(s) |\right)  |f^{\mathrm{in}}(z_{\star}) \di z_\star,
\end{aligned}
\end{align}
where $C_\beta$ is a Lipschitz constant for the communication weight function $\phi$. \newline

Now, we combine \eqref{New-D-5} and \eqref{423} to get 
\begin{align}
\begin{aligned} \label{New-D-7}
& \int_{\mathbb{R}^{2d}} | V_f(t) - V_g(t)|  f^{\mathrm{in}}(z) \di z   \\
& \hspace{0.5cm} \leq \int_0^t  \int_{\mathbb{R}^{2d}}   \Big| L[f](s, X_f(s), V_f(s)) -  L[g](s, X_g(s), V_g(s)) \Big|  f^{\mathrm{in}}(z) \di z  \di s \\
& \hspace{0.5cm} \leq  \kappa  \int_0^t   \int_{\mathbb{R}^{2d}}   |V_{f}(s)-V_{g}(s) | f^{\mathrm{in}}(z) \di z \di s \\
& \hspace{0.7cm}  + \kappa \int_0^t \int_{\mathbb{R}^{4d}} |v_{f,\star}(s)-v_{g,\star}(s) |f^{\mathrm{in}}(z_{\star})   f^{\mathrm{in}}(z) \di z_\star \di z  \di s \\
& \hspace{0.5cm} +  \kappa C_{\beta}
	\int_0^t \int_{\mathbb{R}^{4d}} \Big( |v_{g,\star}(s)|+ |V_{g}(s) | \Big) \\
	& \hspace{0.7cm}  \times  \Big(|x_{f,\star}(s)-x_{g,\star}(s)|+ |X_{f}(s)-X_{g}(s) | \Big) f^{\mathrm{in}}(z_{\star}) f^{\mathrm{in}}(z)   \di z_{\star} \di z \di s.
\end{aligned}
\end{align}
Finally, we combine \eqref{New-D-4} and \eqref{New-D-7} to see
\begin{align}
\begin{aligned} \label{New-D-8}
	\Delta[f,g](t) &= \int_{\mathbb{R}^{2d}} \Big( |X_{f}(t)-X_{g}(t)|+ |V_{f}(t)-V_{g}(t) | \Big) f^{\mathrm{in}}(z) \di  z \\
	&\leq 2  \kappa \int_0^t   \Delta(f,g)(s) \di s +   \kappa \int_0^t \int_{\mathbb{R}^{4d}} |v_{f,\star}(s)-v_{g,\star}(s) |f^{\mathrm{in}}(z_{\star})   f^{\mathrm{in}}(z) \di z_\star\di z  \di s \\
&+  \kappa C_{\beta}
	\int_0^t \int_{\mathbb{R}^{4d}} |v_{g,\star}(s)| |x_{f,\star}(s)-x_{g,\star}(s)|  f^{\mathrm{in}}(z_{\star})   f^{\mathrm{in}}(z) \di z_\star \di z \di s \\
	&+  \kappa C_{\beta}
	\int_0^t \int_{\mathbb{R}^{4d}}  |V_{g}(s)| |x_{f,\star}(s)-x_{g,\star}(s)| f^{\mathrm{in}}(z_{\star}) f^{\mathrm{in}}(z) \di z_\star \di z \di s \\
	& +   \kappa  C_{\beta}\int_0^t \int_{\mathbb{R}^{4d}}  |v_{g,\star}(s)| |X_{f}(s)-X_{g}(s) |  f^{\mathrm{in}}(z_{\star}) f^{\mathrm{in}}(z) \di z_\star \di z \di s \\ 
	& +  \kappa  C_{\beta}\int_0^t \int_{\mathbb{R}^{4d}}  |V_{g}(s) | |X_{f}(s)-X_{g}(s) |  f^{\mathrm{in}}(z_{\star}) f^{\mathrm{in}}(z) \di z_\star \di z \di s \\ 
	& =: 2  \kappa \int_0^t \Delta[f,g](s) \di s + \sum_{i=1}^{5} {\mathcal I}_{1i}.
\end{aligned}
\end{align}
Below, we estimate the term ${\mathcal I}_{1i}$ one by one. \newline

\noindent $\bullet$~Case A.1 (Estimate of ${\mathcal I}_{11}$): We use the Fubini theorem and $\| f^{\mathrm{in}} \|_1 = 1$ to get 
\begin{align}
\begin{aligned} \label{New-D-9}
 {\mathcal I}_{11} &=    \kappa \int_0^t \int_{\mathbb{R}^{4d}} |v_{f,\star}(s)-v_{g,\star}(s) |f^{\mathrm{in}}(z_{\star})   f^{\mathrm{in}}(z) \di z_\star \di z  \di s \\
 & =  \kappa  \int_0^t  \Big( \int_{\mathbb{R}^{2d}}  |v_{f,\star}(s)-v_{g,\star}(s) |f^{\mathrm{in}}(z_{\star})  \di z_\star \Big) \Big( \int_{\bbr^{2d}}  f^{\mathrm{in}}(z) \di z \Big) 
 \di s \\
 &=  \kappa \int_0^t \Delta[f,g](s) \di s.
\end{aligned}
\end{align}

\noindent $\bullet$~Case A.2 (Estimate of ${\mathcal I}_{12}$): For sufficiently large $R>0$, we have
\begin{align}
\begin{aligned} \label{New-D-10}
{\mathcal I}_{12} &=  \kappa C_{\beta}
	\int_0^t \int_{\mathbb{R}^{2d}}  |v_{g,\star}(s)| |x_{f,\star}(s)-x_{g,\star}(s)| f^{\mathrm{in}}(z_{\star}) \di z_\star  \di s \\
	&= \kappa C_{\beta}
	\int_0^t \int_{|v_{g,\star}(s)|> R}  |v_{g,\star}(s)| |x_{f,\star}(s)-x_{g,\star}(s)| f^{\mathrm{in}}(z_{\star}) \di z_\star   \di s \\
&\quad+	\kappa C_{\beta}
	\int_0^t \int_{ |v_{g,\star}(s)| \le R}  |v_{g,\star}(s)| |x_{f,\star}(s)-x_{g,\star}(s)| f^{\mathrm{in}}(z_{\star}) \di z_\star  \di s\\
	&\le  \kappa C_{\beta}
	\int_0^t \int_{|v_{g,\star}(s)|> R}  |v_{g,\star}(s)| |x_{f,\star}(s)-x_{g,\star}(s)| f^{\mathrm{in}}(z_{\star}) \di z_\star   \di s \\
	&\quad+	\kappa C_{\beta}R 
	\int_0^t \int_{|v_{g,\star}(s)| \le R}  |x_{f,\star}(s)-x_{g,\star}(s)| f^{\mathrm{in}}(z_{\star}) \di z_\star  \di s.\\
\end{aligned}
\end{align}

\noindent $\bullet$~Case A.3 (Estimate of ${\mathcal I}_{13}$): By direct estimate, we have
\begin{align}
\begin{aligned} \label{New-D-11}
{\mathcal I}_{13} &=  \kappa C_{\beta} \int_0^t \int_{\mathbb{R}^{4d}}  |V_{g}(s)| |x_{f,\star}(s)-x_{g,\star}(s)| f^{\mathrm{in}}(z_{\star}) f^{\mathrm{in}}(z) \di z_\star \di z \di s \\
& = \kappa C_{\beta} \int_0^t \Big( \int_{\mathbb{R}^{2d}} |x_{f,\star}(s)-x_{g,\star}(s)| f^{\mathrm{in}}(z_{\star}) \di z_\star \Big) \Big( \int_{\mathbb{R}^{2d}}  |V_{g}(s)| f^{\mathrm{in}}(z)    \di z  \Big) \di s.\\
\end{aligned}
\end{align}	
\noindent $\bullet$~Case A.4 (Estimate of ${\mathcal I}_{14}$): By direct estimate, we have
\begin{align}
\begin{aligned} \label{New-D-12}
{\mathcal I}_{14} &=   \kappa C_{\beta}\int_0^t \int_{\mathbb{R}^{4d}}  |v_{g,\star}(s)| |X_{f}(s)-X_{g}(s) |  f^{\mathrm{in}}(z_{\star}) f^{\mathrm{in}}(z) \di z_\star \di z \di s \\ 
&=  \kappa C_{\beta}\int_0^t \Big( \int_{\mathbb{R}^{2d}}   |v_{g,\star}(s)|  f^{\mathrm{in}}(z_{\star})  \di z_\star  \Big) \Big(  \int_{\mathbb{R}^{2d}}   |X_{f}(s)-X_{g}(s) | f^{\mathrm{in}}(z) \di z \Big) \di s.
\end{aligned}
\end{align}
\noindent $\bullet$~Case A.5 (Estimate of ${\mathcal I}_{15}$): By direct estimate, for sufficiently large $R$ we have
\begin{align}
\begin{aligned} \label{New-D-13}
{\mathcal I}_{15} &= \kappa  C_{\beta}\int_0^t \int_{\mathbb{R}^{4d}}  |V_{g}(s) | |X_{f}(s)-X_{g}(s) |  f^{\mathrm{in}}(z_{\star}) f^{\mathrm{in}}(z) \di z_\star \di z \di s \\ 
&=  \kappa C_{\beta}\int_0^t \Big( \int_{\mathbb{R}^{2d}}  |V_{g}(s) | |X_{f}(s)-X_{g}(s)|  f^{\mathrm{in}}(z) \di z      \Big) \di s \\
&= \kappa C_{\beta}\int_0^t \Big( \int_{|V_{g}(s) |>R}  |V_{g}(s) | |X_{f}(s)-X_{g}(s)|  f^{\mathrm{in}}(z) \di z      \Big) \di s \\
&\quad+ \kappa C_{\beta}\int_0^t \Big( \int_{|V_{g}(s) |\le R}  |V_{g}(s) | |X_{f}(s)-X_{g}(s)|  f^{\mathrm{in}}(z) \di z      \Big) \di s \\
&\le  \kappa C_{\beta}\int_0^t \Big( \int_{|V_{g}(s) |>R}  |V_{g}(s) | |X_{f}(s)-X_{g}(s)|  f^{\mathrm{in}}(z) \di z      \Big) \di s \\
&\quad+ \kappa C_{\beta}R\int_0^t \Big( \int_{|V_{g}(s) |\le R}  |X_{f}(s)-X_{g}(s)|  f^{\mathrm{in}}(z) \di z      \Big) \di s.
\end{aligned}
\end{align}
{Then, we collect all the estimates to derive 
\begin{align}
	\Delta[f,g](t)
	&=\int_{\mathbb{R}^{2d}}\Big(|X_f(t)-X_g(t)|+|V_f(t)-V_g(t)|\Big)f^{\mathrm{in}}(z) \di z \notag\\
	&\le C\kappa \int_0^t\int_{|V_g(s)|>R} |V_g(s)|\,|X_f(s)-X_g(s)|\,f^{\mathrm{in}}(z) \di z\di s \notag\\
	&\quad +(C\kappa+2\kappa C_\beta R)\int_0^t \Delta[f,g](s)\di s.
	\label{New-D-14}
\end{align}
For notational simplicity, we set 
\[
A_R:=C\kappa+2\kappa C_\beta R
\]
and
\[
E_R(t):=C\kappa \int_0^t\int_{|V_g(s)|>R} |V_g(s)|\,|X_f(s)-X_g(s)| f^{\mathrm{in}}(z) \di z \di s.
\]
Then, the relation \eqref{New-D-14} can be rewritten as
\[
\Delta[f,g](t)\le A_R\int_0^t \Delta[f,g](s)\di s + E_R(t).
\]
Hence, by the integral form of Gr\"onwall's lemma and the fact that $\Delta[f,g](0)=0$, we have
\begin{equation}
	\Delta[f,g](t)\le e^{A_R t} E_R(t).
	\label{New-D-1510}
\end{equation}
\noindent $\bullet$~Step A.3: By the Cauchy--Schwarz inequality, we have
\begin{align}
	E_R(t)
	&\le C\kappa \int_0^t
	\left(\int_{|V_g(s)|>R} |X_f(s)-X_g(s)|^2 f^{\mathrm{in}}(z)\di z\right)^{\frac12}
	\left(\int_{|V_g(s)|>R} |V_g(s)|^2 f^{\mathrm{in}}(z) \di z\right)^{\frac12}
	\di s.
	\label{New-D-151}
\end{align}
On the one hand, by Lemma \ref{L3.1}, we have
\begin{equation}
	\int_0^t
	\left(\int_{|V_g(s)|>R} |X_f(s)-X_g(s)|^2 f^{\mathrm{in}}(z)\di z\right)^{\frac12}
	\di s
	\le C_\tau.
	\label{New-D-152}
\end{equation}
On the other hand, since the initial datum has an exponential moment and \eqref{NN-3}$_1$ yields the propagation of the exponential velocity moment, we can choose $R$ sufficiently large such that
\[
|v|^2 \le e^{\frac{\alpha}{2}|v|}, \qquad |v|>R.
\]
Therefore, we have
\begin{align*}
\begin{aligned}
& \int_{|V_g(s)|>R} |V_g(s)|^2 f^{\mathrm{in}}(z)\di z \\
& \hspace{0.8cm} \le \int_{|V_g(s)|>R} e^{-\frac{\alpha}{2}R} e^{\frac{\alpha}{2}|V_g(s)|}|V_g(s)|^2 f^{\mathrm{in}}(z)\di z \\
& \hspace{0.8cm}  \le e^{-\frac{\alpha}{2}R}\int_{|V_g(s)|>R} e^{\alpha|V_g(s)|} f^{\mathrm{in}}(z)\di z \\
& \hspace{0.8cm}  \le e^{-\frac{\alpha}{2}R}\int_{\mathbb{R}^{2d}} e^{\alpha|V_g(s)|} f^{\mathrm{in}}(z)\di z \\
& \hspace{0.8cm} \le e^{-\frac{\alpha}{2}R}\int_{\mathbb{R}^{2d}} e^{\alpha|v|} f^{\mathrm{in}}(z)\di z.
\end{aligned}
\end{align*}
Hence, we have
\begin{equation}
	\sup_{0\le s\le \tau}
	\left(\int_{|V_g(s)|>R} |V_g(s)|^2 f^{\mathrm{in}}(z)\di z\right)^{\frac12}
	\le C e^{-\frac{\alpha}{4}R}.
	\label{New-D-153}
\end{equation}
Combining \eqref{New-D-151}, \eqref{New-D-152} and \eqref{New-D-153}, we obtain
\[
E_R(t)\le C_\tau \kappa e^{-\frac{\alpha}{4}R}.
\]
Substituting this into \eqref{New-D-1510}, we get
\[
\Delta[f,g](t)\le C_\tau \kappa \exp\!\left(A_R t-\frac{\alpha}{4}R\right)
= C_\tau \kappa \exp\!\left(C\kappa t+2\kappa C_\beta Rt-\frac{\alpha}{4}R\right).
\]
Thus, for any
\[
0\le t\le T_*:=\frac{\alpha}{16\kappa C_\beta},
\]
we have
\[
2\kappa C_\beta Rt-\frac{\alpha}{4}R
\le -\frac{\alpha}{8}R, \quad 
\Delta[f,g](t)\le C_\tau \kappa \exp\!\left(C\kappa t-\frac{\alpha}{8}R\right).
\]
Letting \(R\to\infty\), we conclude that
\[
\Delta[f,g](t)=0, \qquad \forall~ t\in[0,T_*].
\]
\noindent $\bullet$~Step A.4: Now we iterate the above argument. Since \(\Delta[f,g](t)=0\) on \([0,T_*]\), we have
\[
X_f(T_*,z)=X_g(T_*,z), \qquad V_f(T_*,z)=V_g(T_*,z)
\]
for \(f^{\mathrm{in}}\)-a.e. \(z\in\mathbb{R}^{2d}\). Hence two characteristic flows coincide at time \(T_*\), and we can repeat the same argument on the interval \([T_*,2T_*]\) successively. Since the constants involved depend only on the a priori bounds on \([0,\tau]\), and not on the initial time of the iteration, repeating this argument finitely many times yields
\[
\Delta[f,g](t)=0, \qquad \forall~ t\in[0,\tau].
\]}
This verifies \eqref{New-D-3} and completes the proof.
%
%
%
\section{Emergence of weak flocking}\label{sec:4}
\setcounter{equation}{0}
In this section, we provide the weak flocking estimate of weak solutions with decaying properties in phase space.  		
\subsection{Polynomially decaying distribution class}\label{sec:4.1} 
 Suppose that $\beta, \gamma, l_1, D_1, D_2,$ and $f^{\mathrm{in}}$ satisfy 
{\begin{align}
		\begin{aligned} \label{NNNN-1}
			& 0 \leq \beta < 1, \quad \kappa> 0, \quad  \gamma > 1, \quad  \gamma\beta<1,\quad l_1>1, \quad   D_1 > \frac{2l_1}{(l_1-1)\left(\gamma -1\right)},    \\
			&  D_2\ge \max\left\{D_1,~2l_1 \right\},\quad {\mathcal M}_p(f^{\mathrm{in}}, D_1, D_2) :=  \int_{\mathbb{R}^{2d}} (|x|^{D_1}+ |v|^{D_2}) f^{\mathrm{in}}(z) \di z < \infty, 
		\end{aligned}
\end{align}}
and let $f$ be a global weak solution to \eqref{A-1}. Then, we claim that
\begin{align}
\begin{aligned} \label{D-1}
& (i)~\lim_{t \to \infty} \int_{\mathbb{R}^{2d}} |v-  v^{\mathrm{in}}_c|^2 f(t, z) \di z  = 0: \quad & \mbox{weak velocity alignment}. \\
& (ii)~\sup\limits_{0\le t< \infty} \int_{\mathbb{R}^{2d}} | x - v^{\mathrm{in}}_ct- x^{\mathrm{in}}_c|^2 f(t, z) \di z <\infty: \quad & \mbox{weak spatial cohesion}.
\end{aligned}
\end{align}			
In what follows, we derive \eqref{D-1} one by one.  Let $z=( x, v),~z_\star = ( x_{\star}, v_{\star})$ be two phase points, and we define particle trajectories $(X(t), V(t)$ ad $(X_{\star}(t), V_{\star}(t))$ as solutions to the following system:
	\begin{align}
		\left\{\begin{aligned} \label{D-1-1}
			&\dot X(t)=V(t),\quad  &X(0)= x, \\
			&\dot V(t)=L[f](t,X(t),V(t)),\quad & V(0)= v,
		\end{aligned}\right.
	\end{align}
	and
	\begin{align}
		\left\{\begin{aligned} \label{D-1-2}
			&\dot X_{\star}(t)= V_{\star}(t),\quad  & X_{\star}(0)= x_{\star}, \\
			&\dot V_{\star}(t)=L[f](t,X_{\star}(t),V_{\star}(t)),\quad & V_{\star}(0)= v_{\star}.
		\end{aligned}\right.
	\end{align}
	
\subsubsection{Weak velocity alignment} \label{sec:4.1.1} Since the derivation is very lengthy, we split its proof into several steps. \newline

\noindent $\bullet$~{\bf Step B.1 (Identification of an effective region)}: Recall that the kinetic density $f(t)$ can have a full phase space $\bbr^{2d}$ as its support, hence particles can spread out in phase space. However, thanks to the decaying properties of $f$ in phase space, most of mass will be concentrated around some time-varying region and the rest of mass can be controlled to be small outside this time-varying region. In the sequel, we call this concentration region as an effective region. The motivation for this time-varying region follows from estimates of Corollary \ref{C3.1}:
\[
\int_{B^c_{R_x(t)}  \times \bbr_v^d}   f(t, z) \di z \le \frac{2^{D_1-1} \mathcal{M}_{p}(f^{\mathrm{in}}, D_1, D_2) (1 + t)^{D_1}}{R_x(t)^{D_1}}.
\]
To get the time-decay of $ \int_{B^c_{R_x(t)}  \times \bbr_v^d}   f(t, z) \di z $, we need to choose $R_x(t)$ so that 
\[  \frac{(1 + t)^{D_1}}{R_x(t)^{D_1}} \to 0 \quad \mbox{as $t \to \infty$}. \]
Hence, as long as we choose a super-linearly growing function $R_x(t)$, the total mass of $f$ on the region $B^c_{R_x(t)}  \times \bbr_v^d$ decays to zero as $t \to \infty$. \newline

\noindent For a positive constant $C_1$, we define an increasing time-dependent radius $R_x$:
	\begin{align}\label{NewD-4-4}
		R_x(t) = C_1(1+ t)^\gamma, \qquad \mbox{for some $\gamma > 1$}. \end{align}
Thus, we define effective time-varying cylindrical regions for one and two-particle phase spaces, respectively:
\[ \Omega_{\mbox{eff}}:= B_{R_x(t)} \times \bbr^d, \qquad  \Omega^2_{\mbox{eff}} := \Big( B_{R_x(t)} \times \bbr^d \Big) \times  \Big( B_{R_x(t)} \times \bbr^d  \Big).
\]
Note that the total mass outside the effective region decays to zero algebraically as $t \to \infty$:
\begin{equation} \label{D-1-2-1}
\int_{|x|\ge R_x(t)} f(t, z) \di z \le   \frac{\mathcal{M}_{p}(f^{\mathrm{in}}, D_1, D_2) (1 + t)^{D_{1}}}{\left(R_x(t)\right)^{D_1}} \lesssim \left(1+t\right)^{-D_1(\gamma-1)}.
\end{equation}

\vspace{0.5cm}

\noindent $\bullet$~{\bf Step B.2 (Velocity alignment functional over an effective region)}: Recall a velocity alignment functional introduced in \cite{k1} and Lemma \ref{L2.2} (iv):
\[ \mathcal{L}[f] :=\int_{\mathbb{R}^{2d}} | v- {v^{\mathrm{in}}_c}|^2f(t,z) \di z = \int_{\mathbb{R}^{2d}} |V(t)- v^{\mathrm{in}}_c|^2f^{\mathrm{in}}(z) \di z. \]
Thus, in order to see the time-decay of the term in L.H.S. of the above relation, it suffices to see the temporal decay of the term in the R.H.S..  By Lemma \ref{L3.2}, one can see that 
	\begin{align}
	\begin{aligned} \label{D-1-3}
	&  \dfrac{\di}{\di t} \int_{\mathbb{R}^{2d}} | v- v_c^{\mathrm{in}}|^2f(t,z) \di z = \frac{\di}{\di t}  \int_{\mathbb{R}^{2d}} |V(t)- v^{\mathrm{in}}_c|^2f^{\mathrm{in}}(z) \di z \\
	 & \hspace{1cm} =-\kappa\int_{\mathbb{R}^{4d}} |V(t)- V_{\star}(t)|^2 \phi(|X(t)-X_{\star}(t)|) f^{\mathrm{in}}(z) f^{\mathrm{in}}(z_{\star}) \di z_\star \di z. \\
	\end{aligned}
	\end{align}
To relate the R.H.S. of \eqref{D-1-3} with the velocity alignment functional, we need to take out the factor $ \phi(|X(t)-X_{\star}(t)|)$ by its minimum which is zero. Hence, we need to split the two-phase space into the union of effective regions and the rest of it so that on the effective region, we can take out $ \phi(|X(t)-X_{\star}(t)|)$ by its minimum over the effective region. 

\vspace{0.5cm}

\noindent $\bullet$~{\bf Step B.3 (Estimate of  $ \phi(|X(t)-X_{\star}(t)|)$ over an effective region)}: For a given $t > 0$, we consider two arbitrary particle trajectories $X(t)$ and $X_\star(t)$ in the effective region: 
	\[ |X(t)|\le R_x(t), \quad |X_{\star}(t)|\le R_x(t).\]
	We set 
	\[
	 \underline{\phi}_{R_x}(t) := \phi(2R_x(t))  .
	\]
Then, we use the non-increasing property of $\phi$  to see 	
	\begin{equation} \label{D-2}
	 |X(t)-X_{\star}(t)|\le 2R_x(t) \quad \mbox{and} \quad \inf\limits_{ x, x_{\star}\in B_{R_x(t)}}\phi(|x-x_{\star}|) \ge \phi(2R_x(t))  = \mathcal{O}(1) (1+t)^{-\beta\gamma}.
	\end{equation}
	
	\vspace{0.5cm}
	
\noindent $\bullet$~{\bf Step B.4 (Estimate of the derivative over an effective region)}: We return to \eqref{D-1-3} and rewrite the R.H.S. of \eqref{D-1-3}  into the integral over the effective region and the rest of it. More precisely, we have
	\begin{align}\label{V1}
		\begin{aligned}
			&-\int_{\mathbb{R}^{4d}} |V(t)- V_{\star}(t)|^2 \phi(|X(t)-X_{\star}(t)|) f^{\mathrm{in}}(z) f^{\mathrm{in}}(z_{\star}) \di z_\star \di z \\
			& \hspace{1cm} \leq - \int_{ \Omega^2_{\mbox{eff}}} |V(t)- V_{\star}(t)|^2 \phi(|X(t)-X_{\star}(t)|) f^{\mathrm{in}}(z) f^{\mathrm{in}}(z_{\star}) \di z_\star \di z \\
			& \hspace{1cm} \leq-  \underline{\phi}_{R_x}(t) \int_{\Omega^2_{\mbox{eff}}} |V(t)- V_{\star}(t)|^2 f^{\mathrm{in}}(z) f^{\mathrm{in}}(z_{\star}) \di z_\star \di z \\
			& \hspace{1cm} = - \underline{\phi}_{R_x}(t) \int_{\mathbb{R}^{4d}} |V(t)- V_{\star}(t)|^2 f^{\mathrm{in}}(z) f^{\mathrm{in}}(z_{\star}) \di z_\star \di z  \\
			&  \hspace{1.4cm} + \underline{\phi}_{R_x}(t) \int_{\mathbb{R}^{2d}} \int_{|X_{\star}(t)|> R_x(t)} |V(t)- V_{\star}(t)|^2 f^{\mathrm{in}}(z) f^{\mathrm{in}}(z_{\star}) \di z_\star \di z  \\
			&  \hspace{1.4cm} + \underline{\phi}_{R_x}(t)  \int_{|X(t)|> R_x(t)}  \int_{\mathbb{R}^{2d}} 	 |V(t)- V_{\star}(t)|^2 f^{\mathrm{in}}(z) f^{\mathrm{in}}(z_{\star}) \di z_\star \di z \\
			&\hspace{1.4cm} {- \underline{\phi}_{R_x}(t)  \int_{|X(t)|> R_x(t)}  \int_{|X_{\star}(t)|> R_x(t)} 	 |V(t)- V_{\star}(t)|^2 f^{\mathrm{in}}(z) f^{\mathrm{in}}(z_{\star}) \di z_\star \di z} \\
			&  \hspace{1cm} =: {\mathcal I}_{21} + {\mathcal I}_{22} + {\mathcal I}_{23}+{{\mathcal I}_{24}},
		\end{aligned}
	\end{align}
	where we used the relation:
\begin{align*}
\Omega^2_{\mbox{eff}} &= (B_{R_x(t)} \times \bbr^d) \times (B_{R_x(t)} \times \bbr^d) \\
&=  \bbr^{4d}  - \bbr^{2d} \times  (B^c_{R_x(t)} \times \bbr^d) -  (B^c_{R_x(t)} \times \bbr^d) \times \bbr^{2d} +{(B^c_{R_x(t)} \times \bbr^d)\times(B^c_{R_x(t)} \times \bbr^d)}.    
\end{align*}
In the following lemma, we estimate the terms ${\mathcal I}_{2i}$ in \eqref{V1} one by one. 
\begin{lemma} \label{L4.1}
 Suppose that $\beta, \gamma, l_1, D_1, D_2,$ and $f^{\mathrm{in}}$ satisfy \eqref{NNNN-1}. The terms ${\mathcal I}_{2i}$ in \eqref{V1} satisfy the following estimates:
\begin{align*}
\begin{aligned}
& (i)~ {\mathcal I}_{21} = -2 \underline{\phi}_{R_x(t)} \int_{\mathbb{R}^{2d}} |V(t)-v^{\mathrm{in}}_c|^2f^{\mathrm{in}}(z) \di z. \\
& (ii)~{\mathcal I}_{22} = {\mathcal I}_{23} \leq  \mathcal{O}(1) \left(1+t\right)^{-\frac{(l_1-1)D_1}{l_1}(\gamma-1)-\beta\gamma}.\\
& (iii)~{{\mathcal I}_{24} \le0}.
\end{aligned}
\end{align*}
\end{lemma}
\begin{proof} 
\noindent (i)~We use $\| f \|_{1} = 1$ and Lemma \ref{L2.2} to find 
\[  \int_{\mathbb{R}^{2d}} (V(t)-v^{\mathrm{in}}_c) f^{\mathrm{in}}(z) \di z  = 0.    \]
Furthermore, we have
\begin{align}\label{I11}
	\begin{aligned}
		{\mathcal I}_{21} &= - \underline{\phi}_{R_x}(t) \int_{\mathbb{R}^{4d}} |V(t)- V_{\star}(t)|^2   f^{\mathrm{in}}(z) f^{\mathrm{in}}(z_{\star}) \di z_\star \di z \\
		&=-  \underline{\phi}_{R_x}(t)  \int_{\mathbb{R}^{4d}} |V(t)-v^{\mathrm{in}}_c+v^{\mathrm{in}}_c- V_{\star}(t)|^2 f^{\mathrm{in}}(z) f^{\mathrm{in}}(z_{\star}) \di z_\star \di z   \\
		&=-2  \underline{\phi}_{R_x}(t)  \int_{\mathbb{R}^{2d}} |V(t)-v^{\mathrm{in}}_c|^2f^{\mathrm{in}}(z) \di z.
	\end{aligned}
\end{align}
\noindent (ii)~Note that 
	\begin{align}\label{I2}
		\begin{aligned}
			{\mathcal I}_{22} &=  \underline{\phi}_{R_x}(t)  \int_{\mathbb{R}^{2d}}\int_{|X_{\star}(t)|> R_x(t)} |V(t)- V_{\star}(t)|^2 f^{\mathrm{in}}(z) f^{\mathrm{in}}(z_{\star}) \di z_\star \di z	\\
			&\le 2   \underline{\phi}_{R_x}(t)  \int_{\mathbb{R}^{2d}}\int_{|X_{\star}(t)|> R_x(t)}\Big ( |V_{\star}(t)|^2+ |V(t) |^2 \Big) f^{\mathrm{in}}(z) f^{\mathrm{in}}(z_{\star}) \di z_\star \di z \\
			& =: {\mathcal I}_{221} + {\mathcal I}_{222}.
	\end{aligned}
	\end{align}
Next, we estimate the term ${\mathcal I}_{22i}$ one by one. \newline

\noindent $\bullet$~Case B.1:~We use the Hölder inequality, $D_2\ge 2l_1$, and Lemma \ref{L3.1} to derive 
	\begin{align}
	\begin{aligned} \label{D-3}
	{\mathcal I}_{221}& = 2   \underline{\phi}_{R_x}(t)  \int_{\mathbb{R}^{2d}}\int_{|X_{\star}(t)|> R_x(t)} |V_{\star}(t)|^2 f^{\mathrm{in}}(z) f^{\mathrm{in}}(z_{\star}) \di z_\star \di z \\
	& = 2   \underline{\phi}_{R_x}(t)   \int_{|X_{\star}(t)|> R_x(t)} | V_{\star}(t)|^2 f^{\mathrm{in}}(z_{\star}) \di z_\star \\
	&\le 2  \underline{\phi}_{R_x}(t)   \left(\int_{\mathbb{R}^{2d}} |V_{\star}(t)|^{2l_1} f^{\mathrm{in}}(z_\star) \di z_\star \right)^{\frac{1}{l_1}}\left(\int_{|X_{\star}(t)|> R_x(t)}f^{\mathrm{in}}(z_\star) \di z_\star \right)^{\frac{l_1-1}{l_1}}\\
		&\le 2  \underline{\phi}_{R_x}(t)  \left(\mathcal{M}_p(f^{\mathrm{in}}, D_1 , 2l_1)\right)^{\frac{1}{l_1}}\left(\int_{|X_{\star}(t)|> R_x(t)}f^{\mathrm{in}}(z_\star) \di z_\star \right)^{\frac{l_1-1}{l_1}}.
	\end{aligned}
	\end{align}

\noindent $\bullet$~Case B.2:~We use the Fubini theorem and Lemma \ref{L3.1} to find 	
\begin{align}
	\begin{aligned} \label{D-4}
{\mathcal I}_{222} &=2   \underline{\phi}_{R_x}(t)   \int_{\mathbb{R}^{2d}}\int_{|X_{\star}(t)|> R_x(t)} | V(t)|^2  f^{\mathrm{in}}(z) f^{\mathrm{in}}(z_{\star}) \di z_\star \di z     \\
	&=2  \underline{\phi}_{R_x}(t)  \Big( \int_{|X_{\star}(t)|> R_x(t)}f^{\mathrm{in}}(z_{\star}) \di z_\star \Big) \cdot \Big( \int_{\mathbb{R}^{2d}} |V(t)|^2 f^{\mathrm{in}}(z) \di z \Big) \\
	&\leq 2  \underline{\phi}_{R_x}(t)  {\mathcal M}_p(f^{\mathrm{in}}, 2, 2) \Big( \int_{|X_{\star}(t)|> R_x(t)}f^{\mathrm{in}}(z_{\star}) \di z_\star \Big).
	\end{aligned}
	\end{align}
In \eqref{I2}, we combine  \eqref{D-2}, \eqref{D-3}, \eqref{D-4} and Corollary \ref{C3.1} to find the desired estimate:
	\begin{align}\label{nwe-I2}
		\begin{aligned}
			{\mathcal I}_{22} & \le2  \underline{\phi}_{R_x}(t)  \Bigg[ \left(\mathcal{M}_p(f^{\mathrm{in}},D_1, 2l_1) \right)^{\frac{1}{l_1}}\left(\int_{|X_{\star}(t)|> R_x(t)}f^{\mathrm{in}}(z_\star) \di z_\star \right)^{\frac{l_1-1}{l_1}}\\
			&\quad+ \mathcal{M}_p(f^{\mathrm{in}},2,2) \int_{|X_{\star}(t)|> R_x(t)} f^{\mathrm{in}}(z_\star) \di z_\star \Bigg] \\
		 &\le2  \underline{\phi}_{R_x}(t)  \Bigg[ \left(\mathcal{M}_p(f^{\mathrm{in}}, D_1, 2l_1) \right)^{\frac{1}{l_1}}\left(\frac{2^{{D_1}-1} {\mathcal M}_p(f^{\mathrm{in}}, D_1, 2l_1) (1 + t)^{D_1}}{R_x(t)^{D_1}}\right)^{\frac{l_1-1}{l_1}}\\
		 &\quad+\mathcal{M}_p(f^{\mathrm{in}},2,2) \frac{2^{{D_1}-1} \mathcal{M}_{p}(f^{\mathrm{in}}, D_1, 2l_1) (1 + t)^{D_1}}{R_x(t)^{D_1}} \Bigg] \\
			&\leq  \mathcal{O}\left(\left(1+t\right)^{-\frac{(l_1-1)D_1}{l_1}(\gamma-1)-\beta\gamma}\right) \quad \mbox{for any $D>0$.} 
		\end{aligned}
	\end{align}

	\noindent {(iii)~Note that \[{\mathcal I}_{24}=-\underline{\phi}_{R_x}(t)  \int_{|X(t)|> R_x(t)}  \int_{|X_{\star}(t)|> R_x(t)} 	 |V(t)- V_{\star}(t)|^2 f^{\mathrm{in}}(z) f^{\mathrm{in}}(z_{\star}) \di z_\star \di z\le0.\]}
\end{proof}
In \eqref{V1}, we combine the estimates in Lemma \ref{L4.1} to find 
	\begin{align}\label{V11}
		\begin{aligned}
			&\dfrac{\di}{\di t} \int_{\mathbb{R}^{2d}} |V(t)- v^{\mathrm{in}}_c|^2f^{\mathrm{in}}(z) \di z\\ 
			&\hspace{1cm} \le-2\kappa  \underline{\phi}_{R_x}(t) \int_{\mathbb{R}^{2d}} |V(t)- v^{\mathrm{in}}_c|^2f^{\mathrm{in}}(z) \di z +\mathcal{O}(1) \kappa (1+t )^{-\Big(\frac{(l_1-1)D_1}{l_1}(\gamma-1)+ \beta\gamma\Big)}.\\
		\end{aligned}
	\end{align}
	By Grönwall's lemma, we have 
	\begin{align}\label{V3}
		\begin{aligned}
			&\int_{\mathbb{R}^{2d}} |V(t)- v^{\mathrm{in}}_c|^2f^{\mathrm{in}}(z) \di z \\ 
			&  \hspace{1cm}\le e^{-2\kappa\int_0^t  \underline{\phi}_{R_x}(s)   \di s}\int_{\mathbb{R}^{2d}} | v- v^{\mathrm{in}}_c|^2f^{\mathrm{in}}(z) \di z \\
			&\hspace{1.4cm}+ {\mathcal O}(1)  \int_0^te^{-2 \kappa\int_s^t  \underline{\phi}_{R_x}(\tau) \di \tau} \left(1+s\right)^{-\Big(\frac{(l_1-1)D_1}{l_1}(\gamma-1)+ \beta\gamma\Big)} \di s.\\
			& \hspace{1cm}= e^{ -2 \kappa \int_0^t  \underline{\phi}_{R_x}(s)  \di s}\int_{\mathbb{R}^{2d}} | v- v_c|^2f^{\mathrm{in}}(z) \di z 
		\\
		&\hspace{1.4cm}+ {\mathcal O}(1)  \int_0^{\frac{t}{2}}e^{-2\kappa \int_s^t  \underline{\phi}_{R_x}(\tau) \di \tau} \left(1+s\right)^{-\Big(\frac{(l_1-1)D_1}{l_1}(\gamma-1)+ \beta\gamma\Big)} \di s\\
			& \hspace{1.4cm}+ {\mathcal O}(1)  \int_{\frac{t}{2}}^{t}e^{-2 \kappa\int_s^t  \underline{\phi}_{R_x(\tau)} \di \tau} \left(1+s\right)^{-\Big(\frac{(l_1-1)D_1}{l_1}(\gamma-1)+ \beta\gamma\Big)} \di s\\
			& \hspace{1cm}\le e^{-2\kappa\int_0^t  \underline{\phi}_{R_x}(s) \di s}\int_{\mathbb{R}^{2d}} | v- v^{\mathrm{in}}_c|^2f^{\mathrm{in}}(z) \di z \\
			&\hspace{1.4cm}+ {\mathcal O}(1) e^{-2\kappa \int_{\frac{t}{2}}^t  \underline{\phi}_{R_x}(\tau) \di \tau} \int_0^{\frac{t}{2}} \left(1+s\right)^{-\Big(\frac{(l_1-1)D_1}{l_1}(\gamma-1)+ \beta\gamma\Big)} \di s\\
			& \hspace{1.4cm} +\mathcal{O}(1)  \left(\left(1+\frac{t}{2}\right)^{-\Big(\frac{(l_1-1)D_1}{l_1}(\gamma-1)+ \beta\gamma\Big)}\right)\int_{\frac{t}{2}}^{t}e^{\int_s^t-2 \kappa \underline{\phi}_{R_x}(\tau) \di \tau}\di s\\
			& \hspace{1cm}=:{\mathcal I}_{31}+{\mathcal I}_{32}+{\mathcal I}_{33}.
		\end{aligned}
	\end{align}

	\begin{lemma} \label{L4.2}
	 Suppose that $\beta, \gamma, l_1, D_1, D_2,$ and $f^{\mathrm{in}}$ satisfy \eqref{NNNN-1}.	The terms ${\mathcal I}_{3i}$ in \eqref{V3} satisfy the following estimates:
			\begin{align}
		\begin{aligned} \label{D-6-6}
		&(i)~ |{\mathcal I}_{31}|\le e^{-|\mathcal{O}(1)|  (1+t)^{1-\beta\gamma}}\int_{\mathbb{R}^{2d}} | v- v^{\mathrm{in}}_c|^2f^{\mathrm{in}}(z) \di z.\\
		&(ii)~|{\mathcal I}_{32}|\lesssim {\mathcal O}(1) e^{-|\mathcal{O}(1)|  (1+t )^{1-\beta\gamma}}(1+t)^{1-\Big(\frac{(l_1-1)D_1}{l_1}(\gamma-1)+ \beta\gamma\Big)}.\\
		&(iii)|{\mathcal I}_{33}|\lesssim  (1+t)^{-\frac{(l_1-1)D_1}{l_1}(\gamma-1)}.
	\end{aligned}
	\end{align}
		\end{lemma}
\begin{proof}
	Note that we have the following relations:	
		\[ 0\le\beta\gamma <1,\quad \underline{\phi}_{R_x}(s)= \mathcal{O}(1) (1+s)^{-\beta\gamma} \quad \mbox{and} \quad R_x(t) = C_1 (1 + t)^{\gamma}. \]
	
	\noindent (i) We have 	\begin{align}
	\begin{aligned} \label{NewD-6}
 e^{-2\kappa\int_0^t  \underline{\phi}_{R_x(s)}  \di s}\le e^{-|\mathcal{O}(1)|  (1+t)^{1-\beta\gamma}}. \\
\end{aligned}
\end{align}
This yields the desired estimate.

		\noindent (ii) 	It is easy to verify \begin{align}
	\begin{aligned} \label{NewwD-6}
	&~ e^{-2\kappa \int_{\frac{t}{2}}^t  \underline{\phi}_{R_x}(\tau) \di \tau}\le  {\mathcal O}(1) e^{-|\mathcal{O}(1)|  (1+t )^{1-\beta\gamma}}, \\
	&\int_0^{\frac{t}{2}} \left(1+s\right)^{-\Big(\frac{(l_1-1)D_1}{l_1}(\gamma-1)+ \beta\gamma\Big)} \di s  \\
	& \hspace{0.3cm} =  \begin{cases}
		\ln (1 + \frac{t}{2}), \quad & \mbox{if}~\frac{(l_1-1)D_1}{l_1}(\gamma-1)+ \beta\gamma = 1, \\
		\frac{1}{\frac{(l_1-1)D_1}{l_1}(\gamma-1)+ \beta\gamma -1} \Big(  1 -  (1 + \frac{t}{2})^{1-\Big(\frac{(l_1-1)D_1}{l_1}(\gamma-1)+ \beta\gamma\Big)} \Big), \quad & \mbox{if}~ \frac{(l_1-1)D_1}{l_1}(\gamma-1)+ \beta\gamma \neq 1.
	\end{cases} \\
\end{aligned}
\end{align}
Then, we have 
\[|{\mathcal I}_{32}|\lesssim {\mathcal O}(1) e^{-|\mathcal{O}(1)|  (1+t )^{1-\beta\gamma}}(1+t)^{1-\Big(\frac{(l_1-1)D_1}{l_1}(\gamma-1)+ \beta\gamma\Big)}.\]

	\noindent (iii) Note that 
\begin{equation} \label{NewwwD-6}
\int_s^t  \underline{\phi}_{R_x}(\tau)  \di \tau =  {\mathcal O}(1)  \int_s^t (1 + \tau)^{-\beta  \gamma} \di \tau = \frac{{\mathcal O}(1)}{1-\beta  \gamma } 
	\Big((1 + s)^{1-\beta \gamma} - (1 + t)^{1-\beta \gamma}    \Big).
\end{equation}
	 Since \(\beta\gamma<1\), by the mean value theorem, for \(s\in [t/2,t]\),
\[
(1+t)^{1-\beta\gamma}-(1+s)^{1-\beta\gamma}
\ge (1-\beta\gamma)(1+t)^{-\beta\gamma}(t-s).
\]
Hence,
\begin{align*}
\begin{aligned}
\int_{\frac{t}{2}}^{t}e^{\int_s^t-2 \kappa \underline{\phi}_{R_x}(\tau) \di \tau} \di s  &\lesssim \int_{\frac{t}{2}}^{t}e^{\left(-2\kappa (1+t)^{-\beta\gamma}(t-s)\right)} \di s \\
& \lesssim \int_0^{t/2}\exp\!\left(-2\kappa (1+t)^{-\beta\gamma}u\right) du \lesssim (1+t)^{\beta\gamma}, \quad \beta\gamma<1.
\end{aligned}
\end{align*}
This yields the desired estimate.
\end{proof}
Finally, we combine \eqref{V3} and Lemma \ref{L4.2} to derive 
\begin{equation} \label{D-7}
\int_{\mathbb{R}^{2d}} |V(t)- v^{\mathrm{in}}_c|^2f^{\mathrm{in}}(z) \di z \leq {\mathcal O}(1) e^{-|\mathcal{O}(1)|  (1+t)^{1-\beta\gamma}} + {\mathcal O}(1)  (1+ t)^{-\frac{(l_1-1)D_1}{l_1}(\gamma-1)}. 
\end{equation}
Note that system parameters $\beta, \gamma$ and $D_1$ satisfy the following relations:
\[ 
\beta\gamma<1, \quad 0 <  \frac{(l_1-1)D_1}{l_1}(\gamma-1).
\]
Then, the following relation holds:
\[ 
 e^{-|\mathcal{O}(1)|  (1+t)^{1-\beta\gamma}} \ll (1+ t)^{-\frac{(l_1-1)D_1}{l_1}(\gamma-1)}, \quad \mbox{for}~~t \gg 1, 
\]
and we have the asymptotic weak velocity alignment:
\begin{equation} \label{D-8}
\int_{\mathbb{R}^{2d}} |V(t)- v^{\mathrm{in}}_c|^2f^{\mathrm{in}}(z) \di z \lesssim  {\mathcal O}(1)  (1+ t)^{-\frac{(l_1-1)D_1}{l_1}(\gamma-1)}.
\end{equation}
\subsubsection{Weak spatial cohesion} \label{sec:4.1.2}
In this part, we provide the weak spatial cohesion estimate  $\eqref{D-1}_2$ of a weak solution. First, we use the Cauchy-Schwarz inequality to see
	\begin{align}\label{X1}
		\begin{aligned}
			&\dfrac{\di}{\di t}\int_{\mathbb{R}^{2d}} |X(t)-v^{\mathrm{in}}_ct- x^{\mathrm{in}}_c |^2f^{\mathrm{in}}(z) \di z \\ 
			& \hspace{0.5cm} =\int_{\mathbb{R}^{2d}}2 \Big \langle X(t)-v^{\mathrm{in}}_ct- x^{\mathrm{in}}_c,  V(t)-v^{\mathrm{in}}_c \Big \rangle f^{\mathrm{in}}(z) \di z \\ 
			& \hspace{0.5cm} \le2\left(\int_{\mathbb{R}^{2d}} |X(t)-v^{\mathrm{in}}_ct- x^{\mathrm{in}}_c|^2f^{\mathrm{in}}(z) \di z \right)^{\frac{1}{2}}\left(\int_{\mathbb{R}^{2d}} |V(t)-v^{\mathrm{in}}_c |^2f^{\mathrm{in}}(z) \di z \right)^{\frac{1}{2}}\\
			& \hspace{0.5cm} \le2\left(\int_{\mathbb{R}^{2d}} |X(t)-v^{\mathrm{in}}_ct- x^{\mathrm{in}}_c|^2f^{\mathrm{in}}(z) \di z \right)^{\frac{1}{2}}\left(\int_{\mathbb{R}^{2d}} |V(t)- v^{\mathrm{in}}_c|^2f^{\mathrm{in}}(z) \di z \right)^{\frac{1}{2}}.
		\end{aligned}
	\end{align}
	This yields
\begin{align*} 
\begin{aligned}
& \dfrac{\di}{\di t} \Big( \int_{\mathbb{R}^{2d}} |X(t)-v^{\mathrm{in}}_ct- x^{\mathrm{in}}_c|^2f^{\mathrm{in}}(z) \di z \Big)^{\frac{1}{2}}  \\
& \hspace{1.5cm} \le  \Big( \int_{\mathbb{R}^{2d}} |V(t)- v^{\mathrm{in}}_c|^2f^{\mathrm{in}}(z) \di z \Big)^{\frac{1}{2}} =  {\mathcal O}(1)  (1+ t)^{- \frac{1}{2} \frac{(l_1-1)D_1}{l_1}(\gamma-1)}.
\end{aligned}
\end{align*}
We integrate the above relation over $t$ to get 
\begin{align}
\begin{aligned} \label{D-9}
&  \left(\int_{\mathbb{R}^{2d}} |X(t)-v^{\mathrm{in}}_ct- x^{\mathrm{in}}_c|^2f^{\mathrm{in}}(z) \di z \right)^{\frac{1}{2}} \\
& \hspace{1cm} \leq  \left(\int_{\mathbb{R}^{2d}} |X^{\mathrm{in}}- x^{\mathrm{in}}_c|^2f^{\mathrm{in}}(z) \di z \right)^{\frac{1}{2}} + {\mathcal O}(1) \int_0^t (1+ s)^{- \frac{1}{2}\frac{(l_1-1)D_1}{l_1}(\gamma-1)} \di s.
\end{aligned}
\end{align}
If we choose a large $D_1$ as in \eqref{NNNN-1} such that 
\[  - \frac{1}{2} \frac{(l_1-1)D_1}{l_1}(\gamma-1) < -1, \quad \mbox{i.e.,} \quad D_1 > \frac{2 l_1}{(l_1-1)(\gamma -1)}. \]
Then, the second term in the right-hand side of \eqref{D-9} is integrable, hence we have
\[
  \left(\int_{\mathbb{R}^{2d}} |X(t)-v^{\mathrm{in}}_ct- x^{\mathrm{in}}_c|^2f^{\mathrm{in}}(z) \di z \right)^{\frac{1}{2}} \leq \left(\int_{\mathbb{R}^{2d}} |X^{\mathrm{in}}- x^{\mathrm{in}}_c|^2f^{\mathrm{in}}(z) \di z \right)^{\frac{1}{2}}  + C
\]
for some positive constant $C$. This yields
\[
\int_{\mathbb{R}^{2d}} |X(t)-v^{\mathrm{in}}_ct- x^{\mathrm{in}}_c|^2f^{\mathrm{in}}(z) \di z \leq {\mathcal O}(1) \Big( \int_{\mathbb{R}^{2d}} |X^{\mathrm{in}}- x^{\mathrm{in}}_c|^2f^{\mathrm{in}}(z) \di z + 1 \Big ).
\]
Now, we take a supremum of the above relation over all $t$ to get the weak spatial cohesion:
\[ \sup\limits_{0\le t< \infty} \int_{\mathbb{R}^{2d}} | x - v^{\mathrm{in}}_ct- x^{\mathrm{in}}_c |^2 f(t, z) \di z <\infty.\]
\begin{remark}\label{RE:4.1}
Below, we provide two comments on the first assertion of Theorem \ref{T2.2}. 
\begin{enumerate}
\item
If we choose a common positive exponent $D=\min\{D_1,D_2\}$ such that 
\[D>\inf\limits_{l_1>1}\left(\max \Big \{\frac{2l_1}{(l_1-1)(\gamma -1)}, ~2l_1 \Big \}\right).\]
Then, we can also establish the weak flocking estimates.
\item
{	If we take \(\gamma=\frac{1-\varepsilon}{\beta}\) with \(0<\varepsilon<1-\beta\), then it is enough to assume
	\[
	D=\min\{D_1,D_2\}>
	\inf_{l_1>1}\left(
	\max\left\{
	\frac{2l_1}{l_1-1}\frac{\beta}{1-\varepsilon-\beta},
	\,2l_1
	\right\}
	\right).
	\]
	The above infimum can be computed explicitly. Indeed, balancing the two terms gives
	\[
	\frac{2l_1}{l_1-1}\frac{\beta}{1-\varepsilon-\beta}=2l_1,
	\]
	and hence
	\[
	l_1=1+\frac{\beta}{1-\varepsilon-\beta}.
	\]
	Therefore, we have
	\[
	\inf_{l_1>1}\left(
	\max\left\{
	\frac{2l_1}{l_1-1}\frac{\beta}{1-\varepsilon-\beta},
	\,2l_1
	\right\}
	\right)
	=2+\frac{2\beta}{1-\varepsilon-\beta}=
	\frac{2(1-\varepsilon)}{1-\varepsilon-\beta}.
	\]
	For instance, if \(\beta=0.9\) and \(\varepsilon=0.05\), then the threshold becomes
	\[
	\frac{2(1-\varepsilon)}{1-\varepsilon-\beta}
	=
	38,
	\]
	so that one needs \(D>38\). This reflects the fact that as \(\beta\to1\), the required moment order \(D\) tends to infinity. In contrast, if \(\beta=0.01\) and \(\varepsilon=0.01\), then
	\[
	\frac{2(1-\varepsilon)}{1-\varepsilon-\beta}
	=
	\frac{1.98}{0.98}
	\approx 2.0205,
	\]
	so that it is enough to take \(D>2.0205\). Thus, when \(\beta\to0\), the required exponent \(D\) can be chosen arbitrarily close to \(2\).
}
\end{enumerate}
\end{remark}
\subsection{Exponentially decaying distribution class}\label{sec:4.2}
In this subsection, we first study the weak flocking estimates of weak solution with exponentially decaying initial distributions:
\[  {\mathcal M}_e(f^{\mathrm{in}},\alpha) =  \int_{\bbr^{2d}} e^{\alpha(|x| + |v|)} f^{\mathrm{in}}(z) \di z < \infty. \]
Then, for $\alpha > 0$, the following estimate holds
\[ \int_{\bbr^{2d}} (|x|^D + |v|^D) f^{\mathrm{in}}(z) \di z \leq C  \int_{\bbr^{2d}} e^{\alpha(|x| + |v|)} f^{\mathrm{in}}(z) \di z, \]
for any $D > 0$. Therefore, $f^{\mathrm{in}}$ itself belongs to a polynomially decaying distribution class. Hence, the analysis in the previous subsection can be applied to derive weak flocking estimates when $\beta<1$.  Then weak flocking \eqref{A-5} emerges with super-polynomial decay asymptotically. \newline

Furthermore, if we additionally require that system parameters and initial datum satisfy
{
\begin{align} 
\begin{aligned} \label{NNN-111}
&\delta>1, \quad \beta\in\Bigg[0,\frac{\delta-1}{\delta}\Bigg), \quad \kappa > 0, \quad \alpha > 0, \\
& {\mathcal M}_e(f^{\mathrm{in}}, \alpha, \delta) = \int_{\mathbb{R}^{2d}} e^{\alpha \left(| x|+| v |^{\delta}\right)} f^{\mathrm{in}}(z) \di z < \infty.
\end{aligned}
\end{align}}
Then, we can show that weak flocking \eqref{A-5} emerges at least exponentially fast. Since most analysis in the previous subsection can be done for our current case, we only point out some differences, whenever we need to mention them.
\subsubsection{Estimate of exponential position moment}

We first use the weak flocking estimate to control the growth of the velocity and position along particle trajectories. To simplify the notation, we assume that
\[
x^{\mathrm{in}}_c=v^{\mathrm{in}}_c=0.
\]
By the weak flocking estimate \eqref{D-8}, we have
\begin{equation} \label{D-88}
	\int_0^{\infty}\left(\int_{\mathbb{R}^{2d}} |V(t)|^2f^{\mathrm{in}}(z) \di z\right)^{\frac{1}{2}} \di t \le C_v.
\end{equation}

\begin{lemma}\label{L4.3}
	Suppose that $\beta,\alpha,\delta,\kappa$ and $f^{\mathrm{in}}$ satisfy \eqref{NNN-111}. Then the following velocity estimate holds:
	\[
	\sup_{0\le t <\infty}|V(t)|\le |v|+\kappa C_v.
	\]
\end{lemma}

{\begin{proof}
		For $\varepsilon>0$, set
		\[
		\Psi_\varepsilon(t):=\big(|V(t)|^2+\varepsilon\big)^{\frac12}.
		\]
		Then
		\[
		\frac{\di}{\di t}\Psi_\varepsilon(t)
		=
		\frac{V(t)}{\big(|V(t)|^2+\varepsilon\big)^{\frac12}}\cdot \dot V(t).
		\]
		Using \eqref{B-1-1}, we obtain
		\begin{align*}
			\frac{\di}{\di t}\Psi_\varepsilon(t)
			&=
			-\kappa\int_{\mathbb R^{2d}}\phi(|X(t)-x_{\star}|)
			\frac{V(t)}{\big(|V(t)|^2+\varepsilon\big)^{\frac12}}
			\cdot (V(t)-v_{\star})\,f(t,z_{\star}) \di z_{\star}  \\
			&\le
			\kappa\int_{\mathbb R^{2d}}\phi(|X(t)-x_{\star}|)
			|v_{\star}|\,f(t,z_{\star}) \di z_{\star}  \\
			&\le
			\kappa\int_{\mathbb R^{2d}} |v_{\star}|\,f(t,z_{\star}) \di z_{\star}.
		\end{align*}
		By the push-forward representation of \(f(t)\) and the Cauchy--Schwarz inequality,
		\[
		\int_{\mathbb R^{2d}} |v_{\star}|\,f(t,z_{\star}) \di z_{\star}
		=
		\int_{\mathbb R^{2d}} |V(t,z_{\star})|\,f^{\mathrm{in}}(z_{\star}) \di z_{\star}
		\le
		\left(\int_{\mathbb R^{2d}} |V(t,z)|^2 f^{\mathrm{in}}(z) \di z\right)^{\frac12}.
		\]
		Therefore,
		\[
		\frac{\di}{\di t}\Psi_\varepsilon(t)
		\le
		\kappa
		\left(\int_{\mathbb R^{2d}} |V(t,z)|^2 f^{\mathrm{in}}(z) \di z\right)^{\frac12}.
		\]
		Integrating over \([0,t]\), we get
		\[
		\Psi_\varepsilon(t)
		\le
		\Psi_\varepsilon(0)
		+\kappa\int_0^t
		\left(\int_{\mathbb R^{2d}} |V(s,z)|^2 f^{\mathrm{in}}(z) \di z\right)^{\frac12}\di s.
		\]
		Letting \(\varepsilon\to0\) and using \eqref{D-88}, we obtain
		\[
		|V(t)|
		\le
		|v|
		+\kappa C_v.
		\]
		Taking the supremum over \(t\ge0\) yields
		\[
		\sup_{0\le t<\infty}|V(t)|\le |v|+\kappa C_v.
		\]
	\end{proof}
}

\begin{lemma}\label{L4.4}
{	Suppose that $\beta,\alpha,\delta,\kappa$ and $f^{\mathrm{in}}$ satisfy \eqref{NNN-111}. Then there exists a positive constant \(C\) such that
	\begin{equation}\label{New4.28}
		\int_{\mathbb{R}^{2d}} e^{\alpha | x|} f(t,z) \di z
		\le C\, e^{2\alpha \left(1-\frac{1}{\delta}\right)t^{\frac{\delta}{\delta-1}}},
		\qquad t\gg 1.
	\end{equation}}
\end{lemma}

\begin{proof}
{	By the push-forward representation,
	\[
	\int_{\mathbb{R}^{2d}} e^{\alpha | x|} f(t,z) \di z
	= \int_{\mathbb{R}^{2d}} e^{\alpha | X(t)|} f^{\mathrm{in}}(z) \di z.
	\]
	Moreover,
	\[
	|X(t)|
	\le |x|+\int_0^t |V(s)| \di s.
	\]
	Using Lemma \ref{L4.3}, we obtain
	\[
	|X(t)|\le |x|+\int_0^t (|v|+\kappa C_v)\di s
	=|x|+|v|t+\kappa C_v t.
	\]
	Hence,
	\begin{align}
	\begin{aligned} \label{New4.29}
	& \int_{\mathbb{R}^{2d}} e^{\alpha | X(t)|} f^{\mathrm{in}}(z) \di z \\
	& \hspace{1cm} \le \int_{\mathbb{R}^{2d}} e^{\alpha |x|+\alpha|v|t+\alpha \kappa C_v t} f^{\mathrm{in}}(z) \di z \le e^{\alpha \kappa C_v t} \int_{\mathbb{R}^{2d}} e^{\alpha |x|+\alpha|v|t} f^{\mathrm{in}}(z) \di z.
	\end{aligned}	
	\end{align}
	Furthermore, by Young's inequality with conjugate exponents \(\delta\) and \(\frac{\delta}{\delta-1}\),
	\[
	|v|t \le \frac{1}{\delta}|v|^\delta+\left(1-\frac{1}{\delta}\right)t^{\frac{\delta}{\delta-1}}.
	\]
	Thus,
	\begin{align}
		\int_{\mathbb{R}^{2d}}e^{\alpha |x|+\alpha|v|t} f^{\mathrm{in}}(z) \di z
		&\le \int_{\mathbb{R}^{2d}} e^{\alpha |x|+\frac{\alpha}{\delta}|v|^{\delta}+\alpha\left(1-\frac{1}{\delta}\right)t^{\frac{\delta}{\delta-1}}} f^{\mathrm{in}}(z) \di z.
		\label{New4.30}
	\end{align}
	Combining \eqref{New4.29}, \eqref{New4.30}, and the assumption
	\[
 {\mathcal M}_e(f^{\mathrm{in}}, \alpha, \delta) =	\int_{\mathbb{R}^{2d}} e^{\alpha \left(| x|+| v |^{\delta}\right)} f^{\mathrm{in}}(z) \di z < \infty,
	\]
	we obtain
	\begin{align*}
		\int_{\mathbb{R}^{2d}} e^{\alpha | x|} f(t,z) \di z
		&= \int_{\mathbb{R}^{2d}} e^{\alpha | X(t)|} f^{\mathrm{in}}(z) \di z\\
		&\le e^{\alpha\left(1-\frac{1}{\delta}\right)t^{\frac{\delta}{\delta-1}}+\alpha \kappa C_v t}
		\int_{\mathbb{R}^{2d}} e^{\alpha |x|+\frac{\alpha}{\delta}|v|^{\delta}} f^{\mathrm{in}}(z) \di z\\
		&\le e^{\alpha\left(1-\frac{1}{\delta}\right)t^{\frac{\delta}{\delta-1}}+\alpha \kappa C_v t}
		\int_{\mathbb{R}^{2d}} e^{\alpha |x|+\alpha|v|^{\delta}} f^{\mathrm{in}}(z) \di z.
	\end{align*}
	Since \(\frac{\delta}{\delta-1}>1\), the factor \(e^{\alpha \kappa C_v t}\) can be absorbed by the super-exponential term. Therefore, there exists a positive constant \(C\) such that
	\[
	\int_{\mathbb{R}^{2d}} e^{\alpha | x|} f(t,z) \di z
	\le Ce^{2\alpha \left(1-\frac{1}{\delta}\right)t^{\frac{\delta}{\delta-1}}}.
	\]}
\end{proof}
Now, we define an increasing time-dependent radius $R_x$:
\begin{align}\label{NewD444}
	R_x(t) = C_1(1+ t)^{\frac{\delta}{\delta-1}}, \qquad \mbox{for some $C_1=3 \left(1-\frac{1}{\delta}\right) $}. \end{align}
Thus, we can similarly define effective time-varying cylindrical regions for one and two-particle phase spaces, respectively:
\[ \Omega_{\mbox{eff}}:= B_{R_x(t)} \times \bbr^d, \qquad  \Omega^2_{\mbox{eff}} := \Big( B_{R_x(t)} \times \bbr^d \Big) \times  \Big( B_{R_x(t)} \times \bbr^d  \Big).
\]
Then the total mass outside the effective region decays to zero exponentially as $t \to \infty$:
\begin{equation} \label{NewD121}
	\int_{|x|\ge R_x(t)} f(t, z) \di z \lesssim   \frac{ {\mathcal M}_e(f^{\mathrm{in}}, \alpha, \delta) e^{2\alpha \left(1-\frac{1}{\delta}\right)t^{\frac{\delta}{\delta-1}}}}{e^{\alpha R_x(t)}} \lesssim e^{-\alpha \left(1-\frac{1}{\delta}\right)t^{\frac{\delta}{\delta-1}}}.
\end{equation}
\begin{remark}
	{Note that the exponential decay of
\[	\int_{|x|\ge R_x(t)} f(t, z) \di z \]
	in \eqref{NewD121} is the contrasted difference with the case of polynomially decaying distributions.}
\end{remark}
\subsubsection{Exponential weak flocking estimate}
{Next, we establish exponential weak flocking estimate, i,e, the second assertions in Theorem \ref{T2.2}. Similarly to polynomial decay case, we can use the non-increasing property of $\phi$  to see 	
\begin{equation} \label{NewD-2}
	|X(t)-X_{\star}(t)|\le 2R_x(t) \quad \mbox{and} \quad \underline{\phi}_{R_x}(t)  = \phi(2R_x(t))  = \mathcal{O}(1) (1+t)^{-\beta\frac{\delta}{\delta-1}}.
\end{equation}
Recall that we have 
	\begin{align}
	\begin{aligned} \label{NewD-1-3}
		&  \dfrac{\di}{\di t} \int_{\mathbb{R}^{2d}} | v- v_c^{\mathrm{in}}|^2f(t,z) \di z   \\
		& \hspace{1cm} =-\kappa\int_{\mathbb{R}^{4d}} |V(t)- V_{\star}(t)|^2 \phi(|X(t)-X_{\star}(t)|) f^{\mathrm{in}}(z) f^{\mathrm{in}}(z_{\star}) \di z_\star \di z \\& \hspace{1cm} \le -\kappa  \underline{\phi}_{R_x}(t) \int_{\mathbb{R}^{4d}} |V(t)- V_{\star}(t)|^2 f^{\mathrm{in}}(z) f^{\mathrm{in}}(z_{\star}) \di z_\star \di z  \\
		&  \hspace{1cm} +\kappa \underline{\phi}_{R_x}(t) \int_{\mathbb{R}^{2d}} \int_{|X_{\star}(t)|> R_x(t)} |V(t)- V_{\star}(t)|^2 f^{\mathrm{in}}(z) f^{\mathrm{in}}(z_{\star}) \di z_\star \di z  \\
		&  \hspace{1cm} + \kappa \underline{\phi}_{R_x}(t)  \int_{|X(t)|> R_x(t)}  \int_{\mathbb{R}^{2d}} 	 |V(t)- V_{\star}(t)|^2 f^{\mathrm{in}}(z) f^{\mathrm{in}}(z_{\star}) \di z_\star \di z \\
		& \hspace{1cm} =:{\mathcal I}_{41}+{\mathcal I}_{42}+{\mathcal I}_{43}.
	\end{aligned}
\end{align}}
\begin{lemma}\label{L4.5}
	Suppose that $\beta,\alpha,\delta,\kappa$ and $f^{\mathrm{in}}$ satisfy \eqref{NNN-111}. The terms ${\mathcal I}_{4i}$ in \eqref{NewD-1-3} satisfy the following estimates:
	\begin{align*}
		\begin{aligned}
			& (i)~ {\mathcal I}_{41} = -2\kappa \underline{\phi}_{R_x(t)} \int_{\mathbb{R}^{2d}} |V(t)-v^{\mathrm{in}}_c|^2f^{\mathrm{in}}(z) \di z. \\
			& (ii)~{\mathcal I}_{42} = {\mathcal I}_{43} \leq\mathcal{O}\left(e^{-\alpha\frac{l_1-1}{l_1} \left(1-\frac{1}{\delta}\right)t^{\frac{\delta}{\delta-1}}}\right), \quad l_1>1.\\
		\end{aligned}
	\end{align*}
\end{lemma}
\begin{proof}
	Since (i) is same as in Lemma \ref{L4.1}, we omit it here.\newline 
	
	\noindent (ii)~Note that 
	\begin{align}\label{newI2}
		\begin{aligned}
			{\mathcal I}_{42}
			&\le 2 \kappa  \underline{\phi}_{R_x}(t)  \int_{\mathbb{R}^{2d}}\int_{|X_{\star}(t)|> R_x(t)}\Big ( |V_{\star}(t)|^2+ |V(t) |^2 \Big) f^{\mathrm{in}}(z) f^{\mathrm{in}}(z_{\star}) \di z_\star \di z \\
			& =: {\mathcal I}_{421} + {\mathcal I}_{422}.
		\end{aligned}
	\end{align}
	Next, we estimate the term ${\mathcal I}_{42i}$ one by one. \newline
	
	\noindent $\bullet$~Case B.1:~We use  same argument in Lemma \ref{L4.1} to derive 
	\begin{align}
		\begin{aligned} \label{newD-3}
			{\mathcal I}_{421}& = 2  \kappa  \underline{\phi}_{R_x}(t)  \int_{\mathbb{R}^{2d}}\int_{|X_{\star}(t)|> R_x(t)} |V_{\star}(t)|^2 f^{\mathrm{in}}(z) f^{\mathrm{in}}(z_{\star}) \di z_\star \di z \\
			& = 2  \kappa  \underline{\phi}_{R_x}(t)   \int_{|X_{\star}(t)|> R_x(t)} | V_{\star}(t)|^2 f^{\mathrm{in}}(z_{\star}) \di z_\star \\
			&\le 2  \kappa \underline{\phi}_{R_x}(t)   \left(\int_{\mathbb{R}^{2d}} |V_{\star}(t)|^{2l_1} f^{\mathrm{in}}(z_\star) \di z_\star \right)^{\frac{1}{l_1}}\left(\int_{|X_{\star}(t)|> R_x(t)}f^{\mathrm{in}}(z_\star) \di z_\star \right)^{\frac{l_1-1}{l_1}}\\
			&\le 2  \kappa \underline{\phi}_{R_x}(t)  \left(\mathcal{M}_p(f^{\mathrm{in}}, D_1 , 2l_1)\right)^{\frac{1}{l_1}}\left(\int_{|X_{\star}(t)|> R_x(t)}f^{\mathrm{in}}(z_\star) \di z_\star \right)^{\frac{l_1-1}{l_1}}.
		\end{aligned}
	\end{align}
	
	\noindent $\bullet$~Case B.2:~We use the Fubini theorem and Lemma \ref{L3.1} to find 	
	\begin{align}
		\begin{aligned} \label{newD-4}
			{\mathcal I}_{422} &=2   \kappa \underline{\phi}_{R_x}(t)   \int_{\mathbb{R}^{2d}}\int_{|X_{\star}(t)|> R_x(t)} | V(t)|^2  f^{\mathrm{in}}(z) f^{\mathrm{in}}(z_{\star}) \di z_\star \di z     \\
			&=2  \kappa \underline{\phi}_{R_x}(t)  \Big( \int_{|X_{\star}(t)|> R_x(t)}f^{\mathrm{in}}(z_{\star}) \di z_\star \Big) \cdot \Big( \int_{\mathbb{R}^{2d}} |V(t)|^2 f^{\mathrm{in}}(z) \di z \Big) \\
			&\leq 2  \kappa \underline{\phi}_{R_x}(t)  {\mathcal M}_p(f^{\mathrm{in}}, 2, 2) \Big( \int_{|X_{\star}(t)|> R_x(t)}f^{\mathrm{in}}(z_{\star}) \di z_\star \Big).
		\end{aligned}
	\end{align}
	In \eqref{newI2}, we combine  \eqref{NewD-2}, \eqref{newD-3}, \eqref{newD-4} and \eqref{NewD121} to find the desired estimate:
	\begin{align}\label{newnwe-I2}
		\begin{aligned}
			{\mathcal I}_{42} & \le2 \kappa  \underline{\phi}_{R_x}(t)  \Bigg[ \left(\mathcal{M}_p(f^{\mathrm{in}},D_1, 2l_1) \right)^{\frac{1}{l_1}}\left(\int_{|X_{\star}(t)|> R_x(t)}f^{\mathrm{in}}(z_\star) \di z_\star \right)^{\frac{l_1-1}{l_1}}\\
			&\quad+ \mathcal{M}_p(f^{\mathrm{in}},2,2) \int_{|X_{\star}(t)|> R_x(t)} f^{\mathrm{in}}(z_\star) \di z_\star \Bigg] \\
			&\leq  \mathcal{O}\left(e^{-\alpha\frac{l_1-1}{l_1} \left(1-\frac{1}{\delta}\right)t^{\frac{\delta}{\delta-1}}}\right), \quad l_1>1.
		\end{aligned}
	\end{align}			
\end{proof}
{In \eqref{NewD-1-3}, we combine the estimates in Lemma \ref{L4.5} to find 
\begin{align}\label{NewV11}
	\begin{aligned}
		&\dfrac{\di}{\di t} \int_{\mathbb{R}^{2d}} |V(t)- v^{\mathrm{in}}_c|^2f^{\mathrm{in}}(z) \di z\\ 
		&\hspace{1cm} \le-2\kappa  \underline{\phi}_{R_x}(t) \int_{\mathbb{R}^{2d}} |V(t)- v^{\mathrm{in}}_c|^2f^{\mathrm{in}}(z) \di z +\mathcal{O}(1) \kappa e^{-\alpha\frac{l_1-1}{l_1} \left(1-\frac{1}{\delta}\right)t^{\frac{\delta}{\delta-1}}}.\\
	\end{aligned}
\end{align}
By Grönwall's lemma, $\frac{\delta}{\delta-1}>1$ and \[\underline{\phi}_{R_x}(t) = \phi(2R_x(t))  = \mathcal{O}(1) (1+t)^{-\beta\frac{\delta}{\delta-1}},\] we have the weak velocity alignment:
\begin{align}\label{NewV311}
	\begin{aligned}
		&\int_{\mathbb{R}^{2d}} |V(t)- v^{\mathrm{in}}_c|^2f^{\mathrm{in}}(z) \di z\\ 
		& \hspace{0.5cm} \le e^{-2\kappa \int_0^t  \underline{\phi}_{R_x}(s)  \di s}\int_{\mathbb{R}^{2d}} | v- v^{\mathrm{in}}_c|^2f^{\mathrm{in}}(z) \di z+ \mathcal{O} (1)  \int_0^{t}e^{-2 \kappa\int_s^t  \underline{\phi}_{R_x}(\tau) \di \tau}e^{-\alpha\frac{l_1-1}{l_1} \left(1-\frac{1}{\delta}\right)s^{\frac{\delta}{\delta-1}}} \di s\\
		& \hspace{0.5cm} \le e^{-2\kappa \int_0^t  \underline{\phi}_{R_x}(s) \di s}\left\{\int_{\mathbb{R}^{2d}} | v- v^{\mathrm{in}}_c|^2f^{\mathrm{in}}(z) \di z+ \mathcal{O} (1)  \int_0^{t}e^{2 \kappa\int_0^s  \underline{\phi}_{R_x}(\tau) \di \tau}e^{-\alpha\frac{l_1-1}{l_1} \left(1-\frac{1}{\delta}\right)s^{\frac{\delta}{\delta-1}}} \di s\right\}\\
		& \hspace{0.5cm} \le\mathcal{O} (1)  e^{-2\kappa \int_0^t  \underline{\phi}_{R_x}(s) \di s}\\
		& \hspace{0.5cm}\le {\mathcal O}(1) e^{-|\mathcal{O}(1)|  (1+t )^{1-\beta\frac{\delta}{\delta-1}}}.
	\end{aligned}
\end{align}
Similar to Section \ref{sec:4.1.2}, we have 
\begin{align}\label{X222}
	\begin{aligned}
		\quad\dfrac{\di}{\di t}\left(\int_{\mathbb{R}^{2d}} |X(t)-v^{\mathrm{in}}_ct- x^{\mathrm{in}}_c |^2f^{\mathrm{in}}(z) \di z\right)^{\frac{1}{2}}\le\left(\int_{\mathbb{R}^{2d}} |V(t)- v^{\mathrm{in}}_c |^2f^{\mathrm{in}}(z) \di z\right)^{\frac{1}{2}}.
	\end{aligned}
\end{align}
Note that in \eqref{NNN-111} we have 
\[ 1-\beta\frac{\delta}{\delta-1}>0, \]
then weak spatial cohesion estimate  $\eqref{D-1}_2$ follows from \eqref{NewV311}. We complete the proof of the second assertion in Theorem \ref{T2.2}.}
\subsection{Critical exponent case $\beta=1$}\label{sec:4.3}
We return to the third assertion in Theorem \ref{T2.2}. Suppose that system parameters and  initial datum satisfy
\begin{align*}
\begin{aligned}
&\frac{ \kappa}{\left(\frac{1}{\alpha}+1\right)P_\infty}>2, \quad \alpha > 0, \quad  \beta =1, \quad \quad P_\infty :=\sup\limits_{(x,v)\in{\rm spt}(f^{\mathrm{in}})} |v|<\infty,  \\
&  \int_{\mathbb{R}^{2d}} e^{\alpha | x|} f^{\mathrm{in}}(z) \di z \leq {\mathcal M}_e(f^{\mathrm{in}}, \alpha)  < \infty,
\end{aligned}
\end{align*}
and let $f$ be a weak solution to \eqref{A-1}. Then, we consider the weak velocity alignment and weak spatial cohesion in what follows. 

\subsubsection{Weak velocity alignment} \label{sec:4.3.1}
For an exponentially decaying distribution class, we set
	\[	R_x(t):=\left(\frac{1}{\alpha}+1\right)P_\infty t.\]
Note that the relation \eqref{New-2-1} gives
	\begin{align}\label{C-3-4}
		\int_{|x_{\star}|\ge R_x(t)} f(t,z_{\star}) \di z_{\star}&\le{\mathcal M}_e(f^{\mathrm{in}}, \alpha)e^{\alpha(P_\infty t-R_x(t))}\le{\mathcal M}_e(f^{\mathrm{in}}, \alpha)e^{-P_\infty t}.
	\end{align}
	Then, we use the same argument in Section \ref{sec:4.1.1} to find  
\begin{align}\label{V111}
	\begin{aligned}
		&\quad\dfrac{\di}{\di t}\int_{\mathbb{R}^{2d}} |V(t)- v^{\mathrm{in}}_c|^2f^{\mathrm{in}}(z) \di z\\ 
		& \hspace{1cm} \le-2\kappa \underline{\phi}_{R_x}(t) \int_{\mathbb{R}^{2d}} |V(t)- v^{\mathrm{in}}_c|^2f^{\mathrm{in}}(z) \di z+\mathcal{O}(1) e^{- \frac{P_\infty}{2} t}.\\
	\end{aligned}
\end{align}
By Grönwall's lemma, we have 
\begin{align}\label{V311}
	\begin{aligned}
		&\int_{\mathbb{R}^{2d}} |V(t)- v^{\mathrm{in}}_c|^2f^{\mathrm{in}}(z) \di z\\ 
		& \hspace{0.5cm} \le e^{-2\kappa \int_0^t  \underline{\phi}_{R_x}(s)  \di s}\int_{\mathbb{R}^{2d}} | v- v^{\mathrm{in}}_c|^2f^{\mathrm{in}}(z) \di z+ \mathcal{O} (1)  \int_0^{t}e^{-2 \kappa\int_s^t  \underline{\phi}_{R_x}(\tau) \di \tau}e^{- \frac{P_\infty}{2} s}  \di s\\
		& \hspace{0.5cm} \le e^{-2\kappa \int_0^t  \underline{\phi}_{R_x}(s) \di s}\left\{\int_{\mathbb{R}^{2d}} | v- v^{\mathrm{in}}_c|^2f^{\mathrm{in}}(z) \di z+ \mathcal{O} (1)  \int_0^{t}e^{2 \kappa\int_0^s  \underline{\phi}_{R_x}(\tau) \di \tau}e^{- \frac{P_\infty}{2}s} \di s\right\}\\
		& \hspace{0.5cm} \le\mathcal{O} (1)  e^{-2\kappa \int_0^t  \underline{\phi}_{R_x}(s) \di s} \le \mathcal{O}\left((1+t)^{-\frac{\kappa}{\left(\frac{1}{\alpha}+1\right)P_\infty}}\right).
	\end{aligned}
\end{align}
{Here, we used 
\begin{align*}
e^{-2\kappa\int_0^t  \underline{\phi}_{R_x}(s) \di s}&\le e^{-2\kappa\int_0^t \frac{1}{\left(1+\left(2\left(\frac{1}{\alpha}+1\right)P_\infty  s\right)^2\right)^{\frac{1}{2}} }\di s} \\
&\le\mathcal{O}\left( e^{-\frac{ \kappa}{\left(\frac{1}{\alpha}+1\right)P_\infty}\int_0^t \frac{1}{1+s} \di s}\right) \le\mathcal{O}\left((1+t)^{-\frac{\kappa}{\left(\frac{1}{\alpha}+1\right)P_\infty}}\right)\end{align*}
 in the fourth inequality.}

	\vspace{0.2cm}
	
\subsubsection{Weak spatial cohesion} \label{sec:4.3.2}	
Similar to Section \ref{sec:4.1.2}, we have 
	\begin{align}\label{X22}
		\begin{aligned}
			\quad\dfrac{\di}{\di t}\left(\int_{\mathbb{R}^{2d}} |X(t)-v^{\mathrm{in}}_ct- x^{\mathrm{in}}_c |^2f^{\mathrm{in}}(z) \di z\right)^{\frac{1}{2}}\le\left(\int_{\mathbb{R}^{2d}} |V(t)- v^{\mathrm{in}}_c |^2f^{\mathrm{in}}(z) \di z\right)^{\frac{1}{2}}.
		\end{aligned}
	\end{align}
Then, we can use $\frac{ \kappa}{\left(\frac{1}{\alpha}+1\right)P_\infty}>2$ and \eqref{V311} to derive the desired estimates.

\begin{remark}
	If $\beta < 1$, we can get an exponentially weak flocking behavior, when the velocity variable has compact support and the spatial variable has exponential decay. This can be seen in \eqref{V311}: 
\begin{align}\label{V312}
	\begin{aligned}
		\int_{\mathbb{R}^{2d}} |V(t)- v^{\mathrm{in}}_c|^2f^{\mathrm{in}}(z) \di z&\le\mathcal{O} (1)  e^{-2\kappa \int_0^t  \underline{\phi}_{R_x}(s) \di s} \le \mathcal{O}\left(e^{-2\kappa\mathcal{O}((1+t)^{1-\beta})}\right).
	\end{aligned}
\end{align}
\end{remark}

\section{Conclusion}\label{sec:5}
In this paper, we have shown the propagation of second-order spatial-velocity moments in the KCS model with non-compact spatial-velocity support, which we call weak flocking behavior. Unlike the compact support scenario, the non-compact setting causes significant challenges, such as the communication weight having a zero lower bound and the sup-linear growth of position support, rendering classical methods inapplicable. To address these issues, we derived refined upper-bound estimates for spatial-velocity moments and established the uniqueness of weak solutions via particle trajectories. We further explored two distinct classes of initial distributions characterized by exponential and polynomial decays in phase space. For both cases, we rigorously demonstrated weak flocking behavior, evidenced by the convergence of the second-order velocity moment, centered on the initial average velocity, to zero over time. Simultaneously, the second-order spatial moment around the center of mass remains uniformly bounded, indicating sustained spatial aggregation. Our results also highlight the persistence of weak flocking in the KCS model under non-compact spatial-velocity support, revealing that velocities still converge towards the initial average velocity when the initial distribution exhibits sufficient concentration. However, we can not deal with the critical case $\beta=1$ under polynomially decaying initial distributions. Moreover, the relativistic KCS model with non-compact spatial-velocity support still needs to be explored. {Furthermore, determining the optimal weak flocking rate for given initial data and system parameters is also an interesting problem.} These interesting questions will be left for future work. 
\section*{Conflict of interest statement}
The authors declare no conflicts of interest.

\section*{Data availability statement}
The data supporting the findings of this study are available from the corresponding author upon reasonable request.

\end{document}